\documentclass[letterpaper,11pt,reqno]{amsart}
\usepackage[T1]{fontenc}
\usepackage[utf8]{inputenc}
\usepackage{lmodern}
\usepackage[english]{babel}
\usepackage{amsmath,a4wide}
\usepackage{xfrac}
\usepackage{esint}
\usepackage{stmaryrd,mathrsfs,bm,amsthm,mathtools,yfonts,amssymb,color,braket,booktabs,graphicx,graphics,amsfonts}
\usepackage{latexsym,microtype,indentfirst,hyperref}
\usepackage{xcolor}
\usepackage{courier}

\newtheorem{theorem}{Theorem}[section]
\newtheorem{corollary}[theorem]{Corollary}

\newtheorem{lemma}[theorem]{Lemma}
\newtheorem{proposition}[theorem]{Proposition}

\theoremstyle{definition}
\newtheorem{definition}[theorem]{Definition}
\newtheorem{remark}[theorem]{Remark}

\newtheorem{assumption}[theorem]{Assumption}

\newtheorem{theoremletter}{Theorem}

\numberwithin{equation}{section}
\numberwithin{subsection}{section}

\newcommand{\R}{\mathbb{R}} 
\newcommand{\C}{\mathbb{C}} 

\newcommand{\eps}{\varepsilon} 
\newcommand{\spt}{\mathrm{spt}} 
\newcommand{\dist}{\mathrm{dist}} 
\newcommand{\Ha}{\mathcal{H}} 
\newcommand{\Leb}{\mathcal{L}} 

\newcommand{\pa}{\partial}
\newcommand{\clos}{{\rm clos}}

\newcommand{\mres}{\mathbin{\vrule height 1.6ex depth 0pt width 
0.13ex\vrule height 0.13ex depth 0pt width 1.3ex}}

\newcommand{\abs}[1]{\lvert#1\rvert} 

\newcommand{\V}{\mathbf{V}} 
\newcommand{\IV}{\mathbf{IV}} 
\newcommand{\var}{\mathbf{var}} 
\newcommand{\bG}{\mathbf{G}} 

\newcommand{\cS}{\mathcal{S}}

\newcommand{\rC}{\mathrm{C}} 
\newcommand{\bC}{\mathrm{C}} 
\newcommand{\bbC}{\mathbf{C}} 

\newcommand{\vrho}{\varrho}

\newcommand{\ssubset}{\subset\joinrel\subset}

\title[Dynamical instability of minimal surfaces at flat singular points]{Dynamical instability of minimal surfaces \\ at flat singular points}

\date{\today}

\author{Salvatore Stuvard}
\address{Dipartimento di Matematica, Universit\`a degli Studi di Milano, Via Saldini 50, I-20133 Milano (MI), Italy}
\email{salvatore.stuvard@unimi.it}

\author{Yoshihiro Tonegawa}
\address{Department of Mathematics, Tokyo Institute of Technology, 2-12-1 Ookayama, Meguro-ku, Tokyo 152-8551, Japan}
\email{tonegawa@math.titech.ac.jp}

\begin{document}

\begin{abstract}
Suppose that a countably $n$-rectifiable set $\Gamma_0$ is the support of a multiplicity-one stationary varifold in $\mathbb{R}^{n+1}$ with a point admitting a flat tangent plane $T$ of density $Q \ge 2$. We prove that, under a suitable assumption on the decay rate of the blow-ups of $\Gamma_0$ towards $T$, there exists a \emph{non-constant}, genuinely time-dependend Brakke flow starting with $\Gamma_0$. The result, which applies, in particular, to a large class of (possibly stable) minimal immersions with branch singularities, shows non-uniqueness of Brakke flow under these conditions. Furthermore, it suggests that stationary varifolds which are \emph{dynamically stable}, i.e. stable with respect to mean curvature flow, may be free from flat singularities.\\


\textsc{AMS Math Subject Classification (2020):} 53E10 (primary), 53A10, 49Q05, 49Q15.
\end{abstract}
\maketitle

\tableofcontents

\section{Introduction}
A family of surfaces is said to move by mean curvature flow (abbreviated hereafter as MCF) if the velocity of 
motion is equal to the mean curvature at each point and time. The MCF is one of the simplest geometric 
evolution problems, and it has been studied intensively by numerous researchers over the last few decades. 
In the early stages of the development of the theory of MCF, 
Brakke introduced in \cite{Brakke} a notion of MCF - which is nowadays referred to as the \emph{Brakke flow} - within the framework of geometric measure theory. It is a generalized notion of MCF where the evolving surfaces are not required to be classical, regular submanifolds, but rather \emph{varifolds}, and where the classical parabolic PDE describing the evolution law is replaced by an \emph{ad hoc} inequality which is adapted to the language of varifolds while still being able to capture the geometric features of MCF; see section \ref{sec:notation} for a brief introduction to the subject, and \cite{Brakke,Ton1} for further references. The advantage of such a seemingly abstract approach is that it allows one to describe the evolution by mean curvature of \emph{singular} surfaces (e.g. a moving network of curves in the plane with multiple junction points or a moving cluster of bubbles in the three-dimensional space), as well as to continue the evolution of classical surfaces also after singularities arise. At the same time, a possible drawback is that the solution to Brakke flow for a given initial datum may not be unique in general. 

Recently, the authors of the present paper proved a general theorem concerning the existence of Brakke flows starting from any given closed countably $n$-rectifiable set $\Gamma_0$ in a strictly convex domain in $\mathbb R^{n+1}$ and with the additional property that the (topological) boundary of the evolving varifolds is fixed throughout the flow \cite{ST19}. This existence result gives rise to a number of questions pertaining to the nature of the Brakke flow. One such question to be discussed in the present paper is the following:
\begin{equation} \tag{$\mathcal{Q}$} \label{main question} 
\begin{split}
& \mbox{Does there exist a \emph{stationary} initial datum $\Gamma_0$}\\
& \mbox{admitting a \emph{non-trivial} Brakke flow starting with it?} 
\end{split}
\end{equation}
Here, ``stationary'' means that the first variation of the associated multiplicity one varifold vanishes, and
 ``non-trivial'' means that the flow is genuinely time-dependent: note that a stationary $\Gamma_0$ 
itself is a time-independent Brakke flow with no motion. Thus, the question is equivalent to inquiring about the non-uniqueness of Brakke flow starting from a given stationary $\Gamma_0$. To avoid
instantaneous vanishing, it is also natural to require the continuity
of the surface measures associated to the Brakke flow at $t=0+$. If $\Gamma_0$ is smooth, 
then one expects that all Brakke flows starting with it should 
be trivial as a consequence of the regularity theory for Brakke flows, both in the interior and at the boundary (see \cite{Kasai-Tone,Ton-2,ST_end,DGS} for the former and \cite{Gas} for the latter) and the uniqueness theorem of smooth mean curvature flows, thus it is interesting to focus on stationary $\Gamma_0$ with singularities. In fact, this observation leads to the following refinement of question \eqref{main question}:
\begin{equation} \tag{$\mathcal{Q'}$} \label{true main question} 
\begin{split}
& \mbox{Which types of singularities, if any, of a stationary $\Gamma_0$ \emph{necessarily} determine}\\
& \mbox{the existence of a non-trivial Brakke flow starting with $\Gamma_0$?} 
\end{split}
\end{equation}

The main result of the paper answers affirmatively to question \eqref{main question}, by identifying a type of singularity with the property described in question \eqref{true main question}. The result can be roughly stated as follows (see Theorem \ref{t:main} for the
precise statement).
\begin{theoremletter} \label{thm:main letter}
Suppose that a closed countably $n$-rectifiable set $\Gamma_0$ is stationary, and that there exists 
$x_0\in\Gamma_0$ with the following properties: 
\begin{itemize}
\item[(1)] one of the tangent cones to $\Gamma_0$ at $x_0$ is a flat plane $T$ with multiplicity
$Q\geq 2$, and
\item[(2)] the rescalings $(\Gamma_0-\{x_0\})/r$ locally converge to $T$ at a rate faster than $(\log(1/r))^{-1/2}$
as $r\rightarrow 0^+$. 
\end{itemize}
Then, there exists a non-trivial Brakke flow starting from $\Gamma_0$.
\end{theoremletter}

We remark that, under a natural assumption to be specified later, the Brakke flow obtained in Theorem \ref{thm:main letter} is continuous at $t=0+$. We observe explicitly that the assumption (2) implies, in particular, that the varifold associated with the plane $T$ and carrying multiplicity $Q$ is the \emph{unique} varifold tangent cone to $\Gamma_0$ at $x_0$. There is a plethora of examples of stationary $\Gamma_0$ admitting the kind of singularities described in (1) and (2). Consider, for instance, the class of minimal immersions with a branch point singularity: these are immersed minimal surfaces $\Gamma_0$ in $\R^3$ which admit a parametrization $X \colon \Omega \to \R^3$ of the form
\[
X(z) = {\rm Re}[f(z)]\,,
\]
where $\Omega \subset \R^2 \simeq \mathbb{C}$ is a neighborhood of $z_0=0$, and $f \colon \Omega \to \mathbb C^3$ is a holomorphic curve in $\mathbb C^3$ with $f=(f_1,f_2,f_3) \in \C^3$ satisfying
\[
(f_1')^2 + (f_2')^2 + (f_3')^2 = 0 \quad \mbox{in $\Omega$} \,, \qquad \mbox{and} \qquad f'(0) = 0\,.
\]
In a suitable system of coordinates $(x^1,x^2,x^3)$ of $\R^3$, such a surface is represented by
\begin{eqnarray*}
x^1(z) + i\,x^2(z) &=& (x_0^1+i\,x_0^2) + a\, z^{Q} + {\rm O}(|z|^{Q+1})\,,\\
x^3(z) &=& x_0^3 + {\rm O}(|z|^{Q+1})\,,
\end{eqnarray*}
for some $Q \geq 2$, and thus it behaves like the (multivalued) graph of a complex root: setting $x_0 = (x_0^1,x_0^2,x_0^3)$, the blow-ups $r^{-1}\,(\Gamma_0 - \{x_0\})$ converge to the plane $T = \{x_3=0\}$ with multiplicity $Q$ with a rate ${\rm O}(r^\alpha)$ for some $\alpha > 0$, and thus much faster than the slow logarithmic decay required in (2); see \cite[Section 3.2]{DHS} for a more detailed analysis of the behavior of minimal surfaces near branch points, and Figure \ref{fig:branching} for a graphical representation.
\begin{figure}
\includegraphics[scale=0.25]{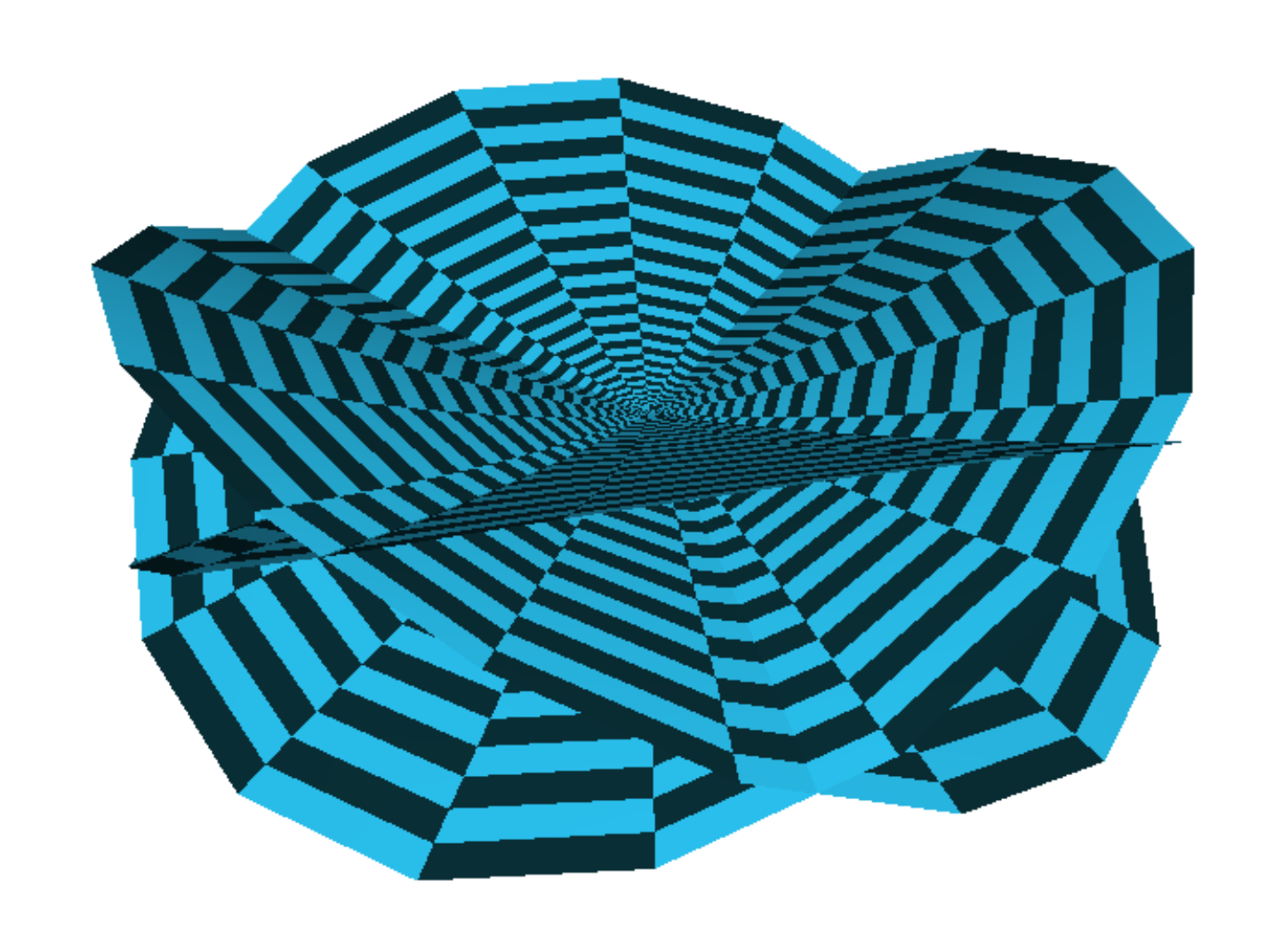} \quad \includegraphics[scale=0.25]{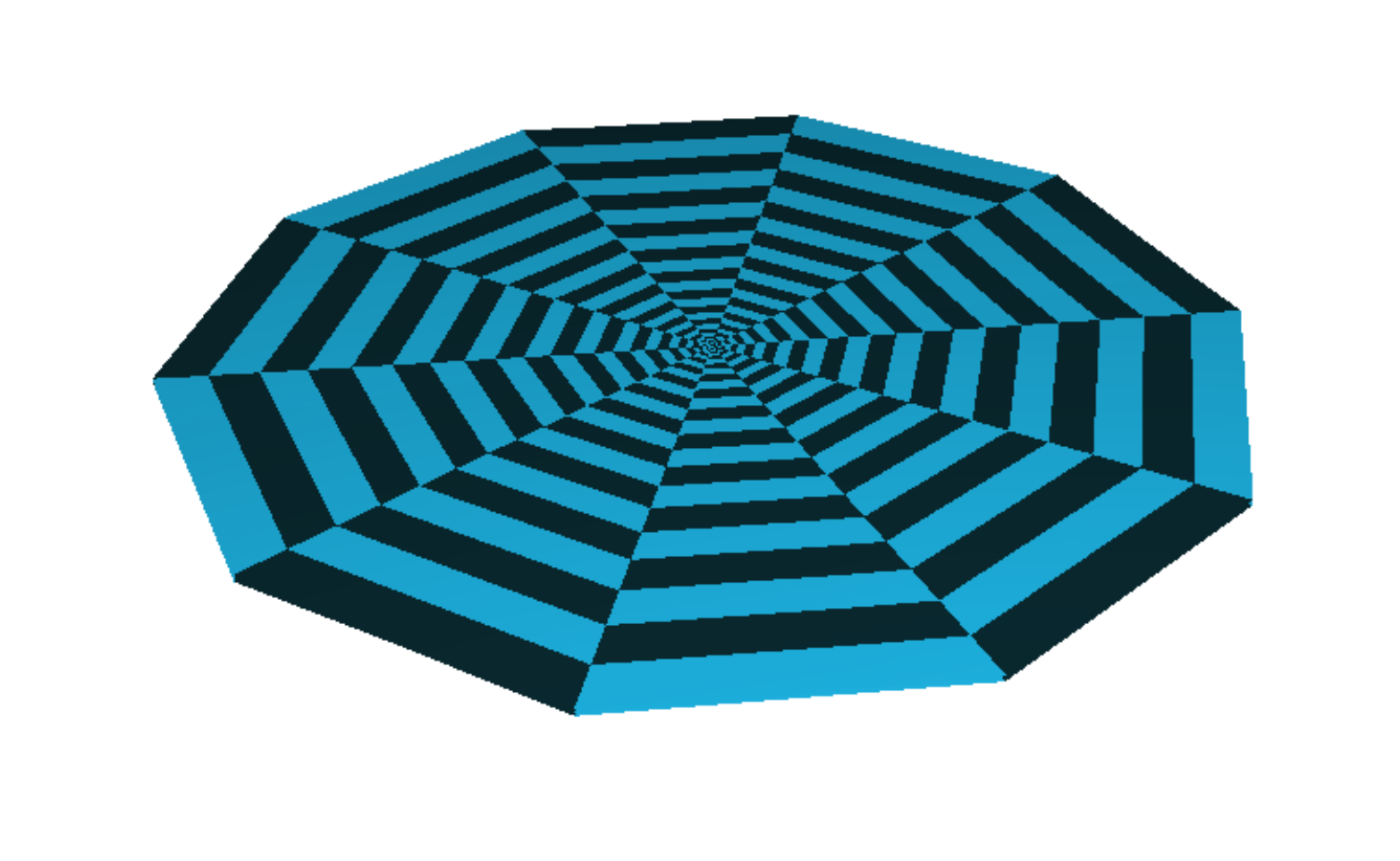}
\caption{\small{On the left, a surface $\Gamma_0$ with a flat singularity of multiplicity $Q=3$. The homothetic rescalings $(\Gamma_0-\{x_0\})/r$ converge to the unique tangent plane with rate ${\rm O}(r^{1/3})$ as $r \to 0^+$. On the right, a figure obtained zooming in the picture on the left at the singularity, thus showing (a portion of) the unique tangent plane $T$ to $\Gamma_0$ at $x_0$.}}\label{fig:branching}
\end{figure}
More generally, Theorem \ref{thm:main letter} applies to the class of those stationary varifolds of arbitrary dimension $n$ that are graphs of multiple valued solutions to the minimal surfaces equation (\textit{$C^{1,\alpha}$ multiple valued minimal graphs}), which have been extensively studied in the literature; see, e.g. \cite{SW07,Ros10,Ros15,Ros16,Krum19,FH20}. \\

Subject to the validity of (1) and (2), Theorem \ref{thm:main letter} concludes the \emph{dynamical instability} of $\Gamma_0$. We warn the reader that dynamical stability/instability is an independent notion with respect to the classical notion of stability/instability defined by the spectrum of the second variation operator on $\Gamma_0$: as observed for instance in \cite{Micallef}, there exist \emph{stable} (yet, by the discussion above and Theorem \ref{thm:main letter}, dynamically unstable) branched minimal surfaces of the type of the disk in $\R^3$. We will come back to give further details on the connection between the two notions of stability after a preliminary discussion concerning the conditions in (1) and (2) and their role in the regularity theory for stationary varifolds, and a review of the existing literature on the topic.

Firstly, when $\Gamma_0$ is stationary, the classical regularity theorem by Allard \cite{Allard} guarantees the existence of a closed set $\cS\subset\Gamma_0$ with $n$-dimensional Hausdorff measure $\mathcal H^n(\cS)=0$ such that $\Gamma_0\setminus \cS$ is an embedded real-analytic minimal hypersurface. With no further assumptions, presently there are no known properties of the singular set $\cS$ other than the fact that it is $\Ha^n$-negligible (except for $n=1$, in which case $\cS$ is a locally $\Ha^0$-finite set, see \cite{AA76}). A crucial missing piece towards a refinement of this result is an estimate on the size of the set of points $x_0\in\Gamma_0$ which admit 
tangent cones that are $n$-dimensional flat planes with multiplicity $Q\geq 2$, that is precisely the points in $\Gamma_0$ satisfying (1). We shall call these points the \emph{flat singularities} of $\Gamma_0$. If one knew that (2) is satisfied at every flat singularity of $\Gamma_0$, then Theorem \ref{thm:main letter} would provide a \emph{dynamical} condition to exclude the presence of flat singularities altogether: explicitly, one would be able to conclude that if $\Gamma_0$ is dynamically stable (that is, if the only Brakke flow starting with $\Gamma_0$ is the trivial one), then no
singular point $x_0 \in \cS$ has a tangent cone which is supported on an $n$-dimensional plane. In turn, this would imply that the singular set $\cS$ of $\Gamma_0$ has Hausdorff dimension $\dim_{\Ha}(\cS) \le n-1$, and in fact that $\cS$ is countably $(n-1)$-rectifiable by the pioneering work of Naber and Valtorta, see \cite{NV_varifolds}. In fact, it would be interesting to investigate the validity of results analogous to that of Theorem \ref{thm:main letter} also under different assumptions on the geometry of the tangent cones to $\Gamma_0$ at a singular point $x_0$ (that is, cones which are not supported on a plane), with the goal of providing further insight into question \eqref{true main question} and correspondingly deducing further information on the fine properties of the singularities of dynamically stable stationary varifolds.

As far as the authors are aware of, there are no known examples of stationary $\Gamma_0$
with a flat singularity for which the decay rate (2) fails. On the other hand, proving that (2) always holds true at flat singularities is arguably a very hard problem, since, as already noticed, it would in particular imply the uniqueness of the tangent plane, which is still a major unsolved problem in geometric measure theory, see \cite[Problem 5.10]{open}. In fact, one may wonder whether a decay rate as in (2) holds true at least assuming \emph{a-priori} that the tangent plane is unique, but even this result is out of reach of the currently available techniques. Indeed, it is worth mentioning that presently all available results concerning uniqueness of tangent cones to a stationary $\Gamma_0$ at a singular point $x_0$ have been obtained under the further assumption that one of the tangent cones to the associated multiplicity one varifold $V_0$ at $x_0$ is a (necessarily non-flat) \emph{multiplicity one} cone $\bbC_0$: for instance, in this setting Simon concluded uniqueness of $\bbC_0$ whenever $\bbC_0$ is regular in $\R^{n+1} \setminus \{0\}$ in \cite{Sim_Loj}, and also when $\bbC_0$ is a cylinder of the form $\bbC_0 = \hat{\bbC}_0 \times \R^{n-k}$ with $\hat{\bbC}_0$ regular in $\R^{k+1}\setminus \{0\}$ under additional hypotheses of integrability of the Jacobi fields of the cross section $\hat{\bbC}_0 \cap \mathbb{S}^k$ and ``absence of holes'' in the singular set of $\Gamma_0$, see \cite{Simon_cylindrical}. In these cases, the multiplicity one assumption is crucial in order to locally parametrize $\Gamma_0$ over $\bbC_0$ with \emph{single-valued} functions, for which PDE techniques are available. It is important to note that both in the cylindrical case treated in \cite{Simon_cylindrical} and in the non-cylindrical case under integrability of Jacobi fields of the cross section (see \cite{AA}), the homothetic rescalings of the varifold $V_0$ at $x_0$ converge towards the unique tangent cone $\bbC_0$ with rate $r^\alpha$ for some $\alpha>0$, that is, the aforementioned parametrization is of class $C^{1,\alpha}$. Recently, a uniqueness result similar to the one of \cite{Sim_Loj} was obtained in the setting of almost area minimizing currents by Engelstein, Spolaor, and Velichkov \cite{ESV}, at the price of producing $C^{1,\log}$ parametrizations. In other words, if an almost area minimizing current has a multiplicity one tangent cone $\bbC_0$ with singularity only at the origin, then the homothetic rescalings of the current converge to $\bbC_0$ at a rate $(\log(1/r))^{-\alpha}$ for some $\alpha>0$ as $r \to 0^+$, see \cite[Theorem 1.5]{ESV}. The similarity between the decay rate of \cite{ESV} and our assumption (2) is interesting, and will be object of further investigation. 

Coming back to flat singularities, much more can be said on whether the condition in (1) implies the decay in (2) if stability of the regular part ${\rm Reg}(\Gamma_0)$ is assumed. Precisely, very recently Minter and Wickramasekera proved in \cite{MW21} the following result: if ${\rm Reg}(\Gamma_0)$ is stable (that is, if every two-sided portion of ${\rm Reg}(\Gamma_0)$ has non-negative second variation with respect to the mass functional for compactly supported normal deformations), if a point $x_0\in\Gamma_0$ satisfies (1) for a plane $T$ and an integer $Q \geq 2$, and furthermore if in a neighborhood $B_{2r_0}(x_0)$ there are no classical singularities of $\Gamma_0$ of density $< Q$ \footnote{A point $x_0 \in \Gamma_0$ is called a \emph{classical singularity} if $\Gamma_0$ is, locally, the union of finitely many, and at least three, embedded $C^1$ submanifolds-with-boundary $M_j$ having the same $(n-1)$-dimensional boundary $\partial M_j = L \ni x_0$ and with $M_i \cap M_j = L$ for $i \neq j$ and with $M_i$ and $M_j$ intersecting transversely at every point in $L$ for at least one pair of indexes $i \neq j$.}, then $T$ is the unique varifold tangent cone to $\Gamma_0$ at $x_0$ and, in the cylinder $((x_0 + T) \cap B_{r_0}(x_0)) \times T^\perp$, $\Gamma_0$ coincides with a (generalized) $C^{1,\alpha}$ $Q$-valued graph, so that the rescalings $r^{-1}\,(\Gamma_0 - \{x_0\})$ converge to $T$ at a rate ${\rm O}(r^\alpha)$ for some $\alpha \in \left(0,1\right)$ depending only on $(n,Q)$; refer to \cite{Alm_big,DLS_Q} for the notion of multiple valued functions used in \cite{MW21}. Theorem \ref{thm:main letter} then immediately implies the dynamical instability of $\Gamma_0$. We will record this result in Corollary \ref{t:Neshan}. 

Notice that, in spite of the fact that the condition on the absence of classical singularities of density $< Q$ in a neighborhood of $x_0$ is, in principle, difficult to guarantee, we point out two classical instances when its validity can be easily checked:
\begin{itemize}
\item[•] if $\Gamma_0$ carries the structure of \emph{rectifiable current}, which we denote $\llbracket \Gamma_0 \rrbracket$, it is stationary as a varifold, ${\rm Reg}(\Gamma_0)$ is stable, and $x_0$ is an \emph{interior} point (that is, $x_0 \notin \spt \|\partial \llbracket \Gamma_0 \rrbracket\|$) satisfying (1) for a plane $T$ and $Q=2$, then (2) holds as a consequence of \cite{MW21}, and $\Gamma_0$ is dynamically unstable; see also \cite{KW21};
\item[•] if $\Gamma_0$ carries the structure of rectifiable current, it is \emph{area minimizing ${\rm mod}(p)$} for an \emph{even} integer $p=2Q$, and it admits a flat tangent plane $T$ at $x_0 \notin \spt\|\partial^p\llbracket\Gamma_0\rrbracket\|$, then the multiplicity of the plane must be $Q$, (2) holds as a consequence of \cite{MW21}, and $\Gamma_0$ is dynamically unstable; see \cite[section 4.2.26]{Federer_book} and \cite{DLHMS_nonlinear} for the definition of area minimizing currents ${\rm mod}(p)$ and the corresponding notation, and \cite{DHMSS_even} for a sharp estimate on the dimension of the set of singularities of area minimizing currents ${\rm mod}(2Q)$ for which (1) holds.  
\end{itemize}
We remark that the above points identify classes of stationary varifolds which are stable for the second variation operator but \emph{not} dynamically stable, further exploring the striking difference between the two notions. \\

The paper is organized as follows. In section \ref{sec:notation} we fix the relevant notation and terminology; in section \ref{sec:statements} we discuss the precise assumptions on the set $\Gamma_0$ and state precisely the main result, Theorem \ref{t:main}; in section \ref{sec:hole} we describe how to suitably modify $\Gamma_0$ in an $\eps$-neighborhood of $x_0$ in order to obtain a new set $\Gamma_0^\eps$ which has strictly less mass than $\Gamma_0$ (we shall say, informally, that $\Gamma_0^\eps$ has a ``hole'' at $x_0$), and then we take advantage of \cite{ST19} to produce a Brakke flow starting with $\Gamma_0^\eps$; sections \ref{sec:expanding holes lemma} and \ref{sec:excess} are the technical core of the paper, as they contain the main estimates needed to show that, along the Brakke flow evolution, the hole in $\Gamma_0^\eps$ \emph{expands} in a precisely quantifiable way. Performing this operation of hole nucleation / hole expansion along a suitable sequence $\eps_j$ produces a sequence of Brakke flows which converges, as $j \to \infty$, to a limiting Brakke flow of surfaces starting with $\Gamma_0$ and having a definite mass drop (with respect to $\Gamma_0$) at a later time, thus completing the proof of Theorem \ref{thm:main letter}: this is achieved in section \ref{sec:main proof}.

\medskip

\noindent\textbf{Acknowledgments:} S.S. was partially supported by grant PRIN 2022PJ9EFL ``\textit{Geometric Measure Theory: Structure of Singular Measures, Regularity Theory and Applications in the Calculus of Variations}'', funded by the European Union under NextGenerationEU and by the Italian Ministry of University and Research \smash{
\begin{tabular}{@{}c@{}}\includegraphics[width=18ex]{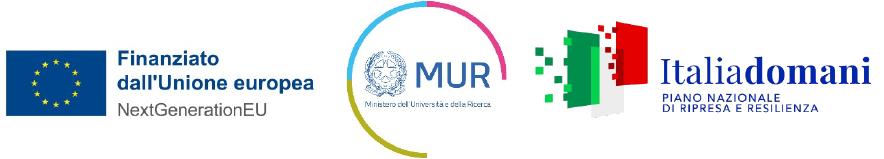}\end{tabular}
}, and by the Gruppo Nazionale per l'Analisi Matematica, la Probabilit\`a e le loro Applicazioni of INdAM; Y.T. was partially supported by JSPS Grant-in-aid for scientific research 18H03670, 19H00639 and 17H01092.

\section{Notation and terminology} \label{sec:notation}

\subsection{Basic notation}
The ambient space we will be working in is Euclidean space $\R^{n+1}$. We write $\R^+$ for $[0,\infty)$.
For $A\subset\mathbb R^{n+1}$, ${\rm clos}\,A$ (or $\overline A$) is the topological closure of $A$ in 
$\mathbb R^{n+1}$, ${\rm int}\,A$ is the set of interior points of $A$ 
and ${\rm conv}\,A$ is the convex hull of $A$. The standard Euclidean inner product between vectors in $\R^{n+1}$ is denoted $x \cdot y$, and $\abs{x} := \sqrt{x \cdot x}$. If $L,S \in \mathscr{L}(\R^{n+1};\R^{n+1})$ are linear operators in $\R^{n+1}$, their (Hilbert-Schmidt) inner product is $L \cdot S := {\rm trace}(L^t \circ S)$, where $L^t$ is the transpose of $L$ and $\circ$ denotes composition. The corresponding (Euclidean) norm in $\mathscr{L}(\R^{n+1};\R^{n+1})$ is then $\abs{L} := \sqrt{L \cdot L}$, whereas the operator norm in $\mathscr{L}(\R^{n+1};\R^{n+1})$ is $\|L\| := \sup\left\lbrace \abs{L(x)} \, \colon \, \mbox{$x\in\R^{n+1}$ with $\abs{x}\leq 1$}  \right\rbrace$. If $u,v \in \R^{n+1}$ then $u \otimes v \in \mathscr{L}(\R^{n+1};\R^{n+1})$ is defined by $(u \otimes v)(x) := (x \cdot v)\, u$, so that $\| u \otimes v \| = \abs{u}\,\abs{v}$. The symbols $U_{r}(x)$ and $B_r(x)$ denote the open and closed balls in $\R^{n+1}$ centered at $x$ and with radius $r > 0$, respectively. The Lebesgue measure of a set $A \subset \R^{n+1}$ is denoted $\Leb^{n+1}(A)$ or $|A|$.  If $1 \leq k \leq n+1$ is an integer, $U_r^k(x)$ denotes the open ball with center $x$ and radius $r$ in $\R^k$. We will set $\omega_k := \Leb^k(U_1^k(0))$. The symbol $\Ha^k$ denotes the $k$-dimensional Hausdorff measure in $\R^{n+1}$, normalized in such a way that $\Ha^{n+1}$ and $\Leb^{n+1}$ coincide as measures. \\

We write $\bG(n+1,k)$ to denote the Grassmannian of (unoriented) $k$-dimensional linear planes in $\R^{n+1}$. Given $T \in \bG(n+1,k)$, we shall often identify $T$ with the orthogonal projection operator onto it, and let $T^\perp := {\rm I} - T$, with ${\rm I}$ the identity operator in $\R^{n+1}$, denote the projection operator onto the orthogonal complement of $T$ in $\R^{n+1}$. If $x \in \R^{n+1}$, $r > 0$, and $T \in \bG(n+1,k)$, then $\bC(T,x,r)$ denotes the cylinder orthogonal to $T$, centered at $x$ with radius $r$, namely the set
\[
\bC(T,x,r) := \left\lbrace  y \in \R^{n+1} \, \colon \, \abs{T(y-x)} < r  \right\rbrace\,.
\]
We will simply write $\bC(x,r)$ in all contexts where the plane $T$ is clear, and $\bC(r)$ when $x=0$.\\

A Radon measure $\mu$ in an open set $U\subset\mathbb R^{n+1}$ is always also regarded as a linear functional on the space $C_c(U)$ of continuous and compactly supported functions on $U$, with the pairing denoted $\mu(\phi)$ for $\phi \in C_c(U)$. The restriction of $\mu$ to a Borel set $A$ is denoted $\mu\, \mres_A$, so that $(\mu \,\mres_A)(E) := \mu(A \cap E)$ for any Borel $E \subset U$. The support of $\mu$ is denoted $\spt\,\mu$, and it is the relatively closed subset of $U$ defined by
\[
\spt\,\mu := \left\lbrace x \in U \, \colon \, \mu(U_r(x)) > 0 \mbox{ for every $r > 0$} \right\rbrace\,.
\]
The upper and lower $k$-dimensional densities of a Radon measure $\mu$ at $x \in U$ are
\[
\Theta^{*k}(\mu,x) := \limsup_{r \to 0^+} \frac{\mu(U_r(x))}{\omega_k\, r^k} \,, \qquad \Theta^k_*(\mu,x) := \liminf_{r \to 0^+} \frac{\mu(U_r(x))}{\omega_k\, r^k}\,,
\]
respectively. If $\Theta^{*k}(\mu,x) = \Theta^k_*(\mu,x)$ then the common value is denoted $\Theta^k(\mu,x)$, and is called the {\it $k$-dimensional density} of $\mu$ at $x$. For $1 \leq p \leq \infty$, the space of $p$-integrable (resp.~locally $p$-integrable) functions with respect to $\mu$ is denoted $L^p(\mu)$ (resp.~$L^p_{{\rm loc}}(\mu)$). For a set $E \subset U$, $\chi_E$ is the characteristic function of $E$. If $E$ is a set of finite perimeter in $U$, then $\nabla \chi_E$ is the associated Gauss-Green measure in $U$, and its total variation $\|\nabla \chi_E\|$ in $U$ is the perimeter measure; by De Giorgi's structure theorem, $\| \nabla \chi_E\| = \Ha^n \mres_{\pa^* E}$, where $\pa^* E$ is the reduced boundary of $E$ in $U$. 
\subsection{Varifolds}
Let $U \subset \R^{n+1}$ be open. The symbol $\V_k(U)$ will denote the space of $k$-dimensional {\it varifolds} in $U$, namely the space of Radon measures on $\bG_k(U) := U \times \bG(n+1,k)$ (see \cite{Allard,Simon} for a comprehensive treatment of varifolds). To any given $V \in \V_k(U)$ one associates a Radon measure $\|V\|$ on $U$, called the {\it weight} of $V$, and defined by projecting $V$ onto the first factor in $\bG_k(U)$, explicitly:
\[
\|V\|(\phi) := \int_{\bG_k(U)} \phi(x) \, dV(x,S) \qquad \mbox{for every $\phi \in C_c(U)$}\,.
\] 
A set $\Gamma \subset \R^{n+1}$ is {\it countably $k$-rectifiable} if it can be covered by countably many Lipschitz images of $\R^k$ into $\R^{n+1}$ up to an $\Ha^k$-negligible set. We say that $\Gamma$ is (locally) {\it $\Ha^k$-rectifiable} if it is $\Ha^k$-measurable, countably $k$-rectifiable, and $\Ha^k(\Gamma)$ is (locally) finite. If $\Gamma \subset U$ is locally $\Ha^k$-rectifiable, and $\theta \in L^{1}_{{\rm loc}}(\Ha^k \mres_\Gamma)$ is a positive function on $\Gamma$, then there is a $k$-varifold canonically associated to the pair $(\Gamma,\theta)$, namely the varifold $\var(\Gamma,\theta)$ defined by
\begin{equation} \label{varGammatheta}
\var(\Gamma,\theta)(\varphi) := \int_\Gamma \varphi(x, T_x \Gamma) \, \theta(x)\, d\Ha^k(x) \qquad \mbox{for every } \varphi \in C_c(\bG_k(U))\,,
\end{equation}
where $T_x\Gamma$ denotes the approximate tangent plane to $\Gamma$ at $x$, which exists $\Ha^k$-a.e. on $\Gamma$. Any varifold $V \in \V_k(U)$ admitting a representation as in \eqref{varGammatheta} is said to be \emph{rectifiable}, and the space of rectifiable $k$-varifolds in $U$ is denoted by 
${\bf RV}_k(U)$. If $V = \var(\Gamma,\theta)$ is rectifiable and $\theta(x)$ is an integer at $\Ha^k$-a.e. $x \in \Gamma$, then we say that $V$ is an \emph{integral} $k$-dimensional varifold in $U$: the corresponding space is denoted $\IV_k(U)$. 

\subsection{First variation of a varifold}
If $V \in \V_k(U)$ and $f \colon U \to U'$ is $C^1$ and proper (that is, $f^{-1}(K)$ is compact for any compact subset $K \subset U'$), then we let $f_\sharp V \in \V_k(U')$ denote the push-forward of $V$ through $f$. Recall that the weight of $f_\sharp V$ is given by
\begin{equation}\label{pushfd}
\|f_\sharp V\|(\phi) = \int_{\bG_{k}(U)} \phi \circ f(x) \, \abs{\wedge_k \nabla f(x) \circ S} \, dV(x,S) \qquad \mbox{for every }\, \phi \in C_{c}(U')\,,
\end{equation}
where
\[
\abs{\wedge_k \nabla f(x) \circ S} := \abs{\nabla f(x) \cdot v_1 \, \wedge \, \ldots \,\wedge\, \nabla f(x) \cdot v_k} \quad \mbox{for any orthonormal basis $\{ v_1, \ldots, v_k \}$ of $S$}
\]
is the Jacobian of $f$ along $S \in \bG(n+1,k)$.
Given a varifold $V \in \V_k(U)$ and a vector field $g \in C^1_c(U; \R^{n+1})$, the {\it first variation} of $V$ in the direction of $g$ is the quantity
\begin{equation}
\label{defFV}
\delta V(g) := \left.\frac{d}{dt}\right|_{t=0} \|(\Phi_t)_\sharp V\|(\tilde U)\,,
\end{equation}
where $\Phi_t(\cdot) = \Phi(t,\cdot)$ is any one-parameter family of diffeomorphisms of $U$ defined for sufficiently small $|t|$ such that $\Phi_0 = {\rm id}_U$ and $\pa_t \Phi(0,\cdot) = g(\cdot)$. The $\tilde U$ is chosen so that ${\rm clos}\,\tilde U\subset U$ is compact and ${\rm spt}\,g\subset \tilde U$, and the definition of \eqref{defFV} 
does not depend on the choice of $\tilde U$. 
It is well known that $\delta V$ is a linear and continuous functional on $C^1_c(U; \R^{n+1})$, and in fact that
\begin{equation}
\label{defFV1}
\delta V(g) = \int_{\bG_k(U)} \nabla g(x) \cdot S \, dV(x,S) \qquad \mbox{for every $g \in C^1_c(U;\R^{n+1})$}\,,
\end{equation}
where, after identifying $S \in \bG(n+1,k)$ with the orthogonal projection operator $\R^{n+1} \to S$,
\[
\nabla g \cdot S = {\rm trace}(\nabla g^t \circ S) = \sum_{i,j=1}^{n+1} S_{ij} \, \frac{\partial g_i}{\partial x_j} =: {\rm div}^S g
\]
is the tangential divergence of $g$ along $S$. If $\delta V$ can be extended to a linear and continuous functional on
$C_c(U;\R^{n+1})$, we say that $V$ has {\it bounded first variation} in $U$. In this case, 
$\delta V$ is naturally associated with 
a unique $\R^{n+1}$-valued measure on $U$ by means of the Riesz representation theorem.
If such a measure is absolutely continuous with respect to the weight $\|V\|$, then there exists a $\|V\|$-measurable and locally $\|V\|$-integrable vector
field $h(\cdot,V)$ such that 
\begin{equation} \label{def:generalized mean curvature}
\delta V(g) = - \int_{U} g(x) \cdot h(x,V) \, d\|V\|(x) \qquad \mbox{for every $g \in C_c(U;\R^{n+1})$}
\end{equation}
by the Lebesgue-Radon-Nikod\'ym differentiation theorem. The vector field $h(\cdot,V)$ is called the {\it generalized mean curvature vector} of $V$. For any $V\in {\bf IV}_k(U)$ with generalized mean curvature $h(\cdot, V)$, {\it Brakke's perpendicularity theorem} \cite[Chapter 5]{Brakke}
says that 
\begin{equation}
\label{BPT}
S^{\perp}(h(x,V))=h(x,V) \qquad \mbox{for $V$-a.e. $(x,S) \in {\bf G}_k(U)$}\,.
\end{equation}
This means that the generalized mean curvature vector is perpendicular to the 
approximate tangent plane almost everywhere. A special mention is due to integral varifolds $V$ for which $h(\cdot,V)=0$ $\|V\|$-almost everywhere: such a varifold will be called {\it stationary}. If $V$ is stationary in $U$ and $x \in \spt(\|V\|)$, then the function $r \in \left( 0, \dist(x,\pa U) \right) \mapsto (\omega_k r^k)^{-1} \|V\|(U_r(x))$ is increasing, so that the density $\Theta_V(x) := \Theta^k(\|V\|,x)$ exists at every $x \in \spt(\|V\|)$. Furthermore, for every sequence $r_h \to 0^+$ there are a subsequence $r_{h'}$ and a stationary integral $k$-varifold $\bbC$ in $\R^{n+1}$ such that, setting $\eta_{x,r}(y):= r^{-1}\,(y-x)$, the varifolds $(\eta_{x,r_{h'}})_\sharp V$ converge to $\bbC$ in the sense of Radon measures on $\bG_k(\R^{n+1})$ as $h' \to \infty$. The varifold $\bbC$ will be called a \emph{tangent cone} to $V$ at $x$, a terminology justified by the homogeneity property $(\eta_{0,\lambda})_\sharp \bbC=\bbC$ for all $\lambda > 0$. \\

Other than the first variation $\delta V$ discussed
above, we shall also use a {\it weighted first variation}, defined as follows. For given $\phi\in C^1_c(U;\mathbb R^+)$, $V\in {\bf V}_k(U)$, and $g \in C^1_c(U;\R^{n+1})$, we modify \eqref{defFV} to introduce the $\phi$-weighted first variation of $V$ in the direction of $g$, denoted $\delta(V,\phi)(g)$, by setting
\begin{equation} \label{defFV_modified}
\delta(V,\phi)(g) := \left.\frac{d}{dt}\right|_{t=0} \| (\Phi_t)_\sharp V \|(\phi)\,,
\end{equation}
where $\Phi_t$ denotes the one-parameter family of diffeomorphisms of $U$ induced by $g$ as above. Proceeding as in the derivation of \eqref{defFV1}, one then obtains the expression
\begin{equation}
\label{defFV2}
\delta(V,\phi)(g)=\int_{{\bf G}_k(U)} \phi(x)\, \nabla g(x)\cdot S\,dV(x,S)+
\int_U g(x)\cdot\nabla\phi(x)\,d\|V\|(x)\,.
\end{equation}
Using $\phi\nabla g=
\nabla(\phi g)- g\otimes\nabla\phi$ in \eqref{defFV2} and \eqref{defFV1}, we obtain
\begin{equation}
\label{defFV3}
\begin{split}
\delta(V,\phi)(g)&=\delta V(\phi g)+\int_{{\bf G}_k(U)} g(x)\cdot(\nabla\phi(x)-S(
\nabla\phi(x)))\,dV(x,S) \\
&=\delta V(\phi g)+\int_{{\bf G}_k(U)} g(x)\cdot S^{\perp}(\nabla\phi(x))\,dV(x,S)\,. 
\end{split}
\end{equation}
If $\delta V$ has generalized mean curvature $h(\cdot,V)$, then we may use \eqref{def:generalized mean curvature} in \eqref{defFV3} to obtain
\begin{equation}
\label{defFV4}
\delta(V,\phi)(g)=-\int_U \phi(x)g(x)\cdot h(x,V) \, d\|V\|(x)+\int_{{\bf G}_k(U)} g(x)\cdot S^{\perp}
(\nabla\phi(x))\,
dV(x,S).
\end{equation}
The definition of Brakke flow requires considering weighted first variations in the direction of the mean curvature. Suppose $V\in {\bf IV}_k(U)$, $\delta V$ is locally bounded and absolutely continuous
with respect to $\|V\|$ and $h(\cdot,V)$ is locally square-integrable with respect to $\|V\|$. 
In this case, it is natural from the expression \eqref{defFV4} to define for $\phi\in C_c^1
(U;\mathbb R^+)$
\begin{equation}
\label{defFV5}
\delta(V,\phi)(h(\cdot,V)):=\int_U \lbrace -\phi(x)\,|h(x,V)|^2+h(x,V)\cdot\nabla\phi(x) \rbrace\,d\|V\|(x).
\end{equation}
Observe that here we have used \eqref{BPT} in order to replace the term $h(x,V)\cdot S^{\perp}(\nabla\phi(x))$ with $h(x,V)\cdot \nabla\phi(x)$.

\subsection{Brakke flow}
In order to motivate the weak formulation of the MCF introduced by Brakke in \cite{Brakke}, note that a smooth family of $k$-dimensional
surfaces $\{\Gamma(t)\}_{t\geq 0}$ in $U$ is a MCF if and only if the following inequality
holds true for all $\phi = \phi (x,t)\in C_c^1(U\times[0,\infty);\mathbb R^+)$:
\begin{equation}
\label{smMCF1}
\frac{d}{dt}\int_{\Gamma(t)}\phi\,d\mathcal H^k \leq \int_{\Gamma(t)}  
\left\lbrace -\phi\,|h(\cdot,\Gamma(t))|^2+\nabla\phi\cdot h(\cdot,\Gamma(t))
+\frac{\partial\phi}{\partial t} \right\rbrace \,d\mathcal H^k \,.
\end{equation}
In fact, the ``only if''
part holds with equality in place of inequality. For a more comprehensive treatment of the Brakke flow, 
see \cite[Chapter 2]{Ton1}. Formally, if $\partial\Gamma(t)\subset
\partial U$ is fixed in time, with $\phi=1$, we also obtain
\begin{equation}
\label{smMCF2}
\frac{d}{dt}\mathcal H^k(\Gamma(t))  \leq -\int_{\Gamma(t)}|h(x,\Gamma(t))|^2\,d\mathcal H^k(x)\,,
\end{equation}
which states the well-known fact that the $L^2$-norm of the mean curvature represents the dissipation of area along the MCF. Motivated by \eqref{smMCF1} and \eqref{smMCF2}, we have defined in \cite{ST19} the following notion of Brakke flow with fixed boundary.
\begin{definition} \label{def:Brakke_bc}
Let $U \subset \R^{n+1}$ be an open set. We say that a family of varifolds $\{V_t\}_{t\geq 0}$ in $U$ is a {\it $k$-dimensional Brakke flow in $U$} if all of the following hold: 
\begin{enumerate}
\item[(a)]
For a.e.~$t\geq 0$, $V_t\in{\bf IV}_k(U)$;
\item[(b)]
For a.e.~$t\geq 0$, $\delta V_t$ is bounded and absolutely continuous with respect to $\|V_t\|$;
\item[(c)] The generalized mean curvature $h(x,V_t)$ (which exists for a.e.~$t$ by (b)) satisfies for all $s>0$
\begin{equation}
\|V_s\|(U)+\int_0^{s}dt\int_U|h(x,V_t)|^2\,d\|V_t\|(x)\leq \|V_0\|(U);
\label{brakineq2}
\end{equation}
\item[(d)]
For all $0\leq t_1<t_2<\infty$ and $\phi\in C_c^1(U\times\R^+;\mathbb R^+)$,
\begin{equation}
\label{brakineq}
\|V_{t}\|(\phi(\cdot,t))\Big|_{t=t_1}^{t_2}\leq \int_{t_1}^{t_2}\delta(V_t,\phi(\cdot,t))(h(\cdot,V_t))+\|V_t\|\big(\frac{\partial\phi}{\partial t}(\cdot,t)\big)\,dt\,,
\end{equation}
having set $\|V_{t}\|(\phi(\cdot,t))\Big|_{t=t_1}^{t_2} := \|V_{t_2}\|(\phi(\cdot,t_2)) - \|V_{t_1}\|(\phi(\cdot,t_1))$.
\end{enumerate}
Furthermore, if $\pa U$ is not empty and $\Sigma \subset \pa U$, we say that $\{V_t\}_{t\ge 0}$ has {\it fixed boundary $\Sigma$} if, together with conditions (a)-(d) above, it holds 
\begin{itemize}
\item[(e)]
For all $t\geq 0$, $ ({\rm clos}\,({\rm spt}\,\|V_t\|))\setminus U=\Sigma$.
\end{itemize}
\end{definition}
\noindent

Notice that, formally, we obtain the analogue of 
\eqref{brakineq2}  by integrating \eqref{smMCF2} from $0$ to $s$. By integrating \eqref{smMCF1} from $t_1$ to $t_2$, we also obtain 
the analogue of \eqref{brakineq} via the expression \eqref{defFV5}. We recall that the
closure is taken with respect to the topology of $\mathbb R^{n+1}$ while the support 
of $\|V_t\|$ is in $U$. Thus (e) geometrically
means that ``the boundary of $V_t$ (or $\|V_t\|$) is $\Sigma$''. 

\section{Main results} \label{sec:statements}

As anticipated in the introduction, as an \emph{initial datum} we are going to consider a closed countably $n$-rectifiable set $\Gamma_0$ in $\R^{n+1}$. In order to guarantee the existence of a Brakke flow starting with $\Gamma_0$ we are going to require that $\Gamma_0$ satisfies the same set of assumptions under which the theory in \cite{ST19} was developed. For the reader's convenience, we record those assumptions here.

\begin{assumption} \label{ass:st19}
Let us fix integers $n \ge 1$ and $N \ge 2$. We consider $U$, $\Gamma_0$, and $\{E_{0,i}\}_{i=1}^N$ such that:
\begin{itemize}
\item[(A1)] $U \subset \R^{n+1}$ is a strictly convex bounded domain with boundary $\pa U$ of class $C^2$;
\item[(A2)] $\Gamma_0 \subset U$ is a relatively closed, countably $n$-rectifiable set with $\Ha^n(\Gamma_0) < \infty$;
\item[(A3)] $E_{0,1}, \ldots, E_{0,N}$ are non-empty, open, and mutually disjoint subsets of $U$ such that $U \setminus \Gamma_0 = \bigcup_{i=1}^N E_{0,i}$;
\item[(A4)] $\pa \Gamma_0 := \clos(\Gamma_0) \setminus U$ is not empty, and for each $x \in \pa\Gamma_0$ there exist at least two indexes $i_1 \ne i_2$ in $\{1,\ldots,N\}$ such that $x \in \clos\left(\clos(E_{0,i_j}) \setminus (U \cup \pa \Gamma_0)\right)$ for $j=1,2$;
\item[(A5)] $\Ha^n(\Gamma_0 \setminus \bigcup_{i=1}^N \pa^*E_{0,i})=0$.
\end{itemize}
\end{assumption}

\begin{remark} \label{rmk:the existence theory}
Under the validity of assumptions (A1)-(A4), there exists a Brakke flow $\{V_t\}_{t \ge 0}$ with fixed boundary $\pa \Gamma_0$ and such that $\|V_0\|=\Ha^n\mres_{\Gamma_0}$, and if also (A5) holds then the surface measures associated to such Brakke flow are continuous at $t=0^+$, that is also $\lim_{t\to 0^+} \|V_t\|=\Ha^n \mres_{\Gamma_0}$; see \cite[Theorem 2.2]{ST19}. Moreover, $\{V_t\}_{t \geq 0}$ can be made canonical in the sense of \cite{ST_can}. As it will become apparent in the sequel, the construction of the present paper is purely local, and based at a fixed point $x_0 \in \Gamma_0$. Hence, the fact that the boundary $\pa \Gamma_0$ is kept fixed throughout the evolution is not important here, and we could potentially also work in the setting of \cite{KimTone}, where $U$ is replaced by the whole Euclidean space $\R^{n+1}$, the finiteness of the $\Ha^n$-measure of $\Gamma_0$ can be assumed to hold locally, and (A4) is dropped. Nonetheless, in order to fix the ideas we will always work in the ``constrained'' fixed boundary case, and leave to the reader the necessary modifications to treat the ``unconstrained'' case. 
\end{remark}

Next, we focus on the main assumption of this paper.

\begin{assumption} \label{assumptions:complete}
Let $U$, $\Gamma_0$, and $\{E_{0,i}\}_{i=1}^N$ satisfy Assumption \ref{ass:st19}, and let $V_0 := \var(\Gamma_0,1)$.
We suppose that
\begin{itemize}
\item[(H0)] $V_0$ is a stationary varifold, with $\spt(\| V_0 \|) = \Gamma_0$, 
\end{itemize}
and that there exists a point $x_0 \in \Gamma_0$, without loss of generality $x_0=0$, with the following properties:
\begin{itemize}
\item[(H1)] one of the tangent cones to $V_0$ at $x_0=0$ is of the form $\var(T,Q)$, for some $n$-dimensional plane $T \in \bG(n+1,n)$ and an integer $Q \ge 2$;

\item[(H2)] there exists a radius $r_0 \in \left( 0,1 \right)$ such that, writing $x = (x',x_{n+1}) \in \R^{n+1} = T \oplus T^\perp$ we have 
\begin{equation} \label{growth}
\Gamma_0 \cap \bC(T,0,r_0) \cap \{|x_{n+1}|<r_0\}\subset \{  x=(x',x_{n+1}) \in \R^{n+1} \,\colon\, \abs{x_{n+1}} \leq G(x')         \}\,,
\end{equation}
 where $G$ is the positive, radial function $G(x') = g(\abs{x'})$ defined by
\begin{equation} \label{our most affordable g}
g(s) = \frac{s}{\log^\alpha\left(1/s\right)} \qquad \mbox{for $s > 0$, with $\alpha > \frac12$}\,,
\end{equation}
see Figure \ref{figure growth}.
\end{itemize}
\end{assumption}

\begin{figure}
\includegraphics[scale=0.25]{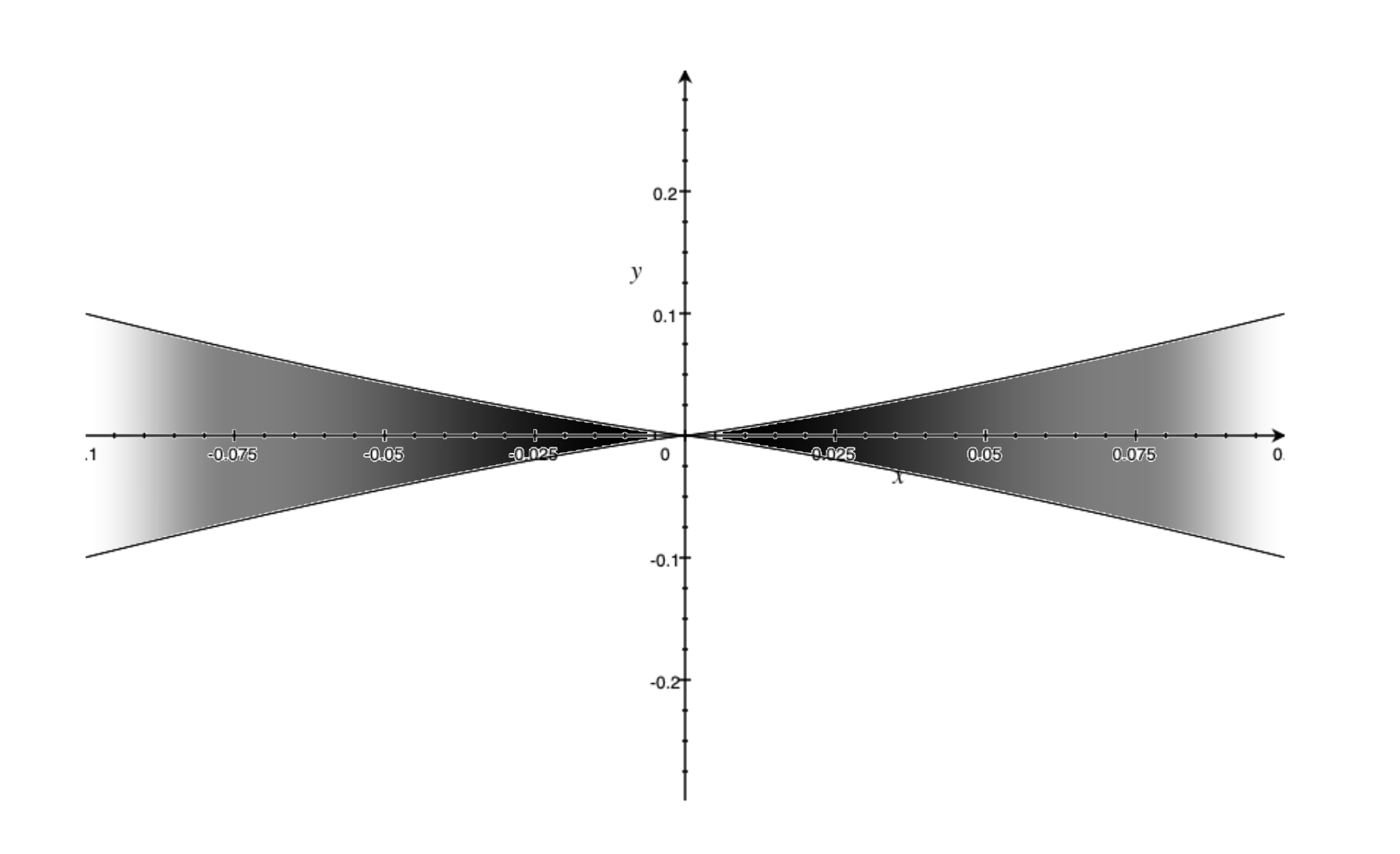} 
\caption{\small{The growth condition required in (H2): the shaded area is the region $\{\abs{x_{n+1}} \leq G(x')\}$ ($\alpha=0.51$) in the cylinder $\rC(T,0,r_0)$ with $r_0=0.1$.}}\label{figure growth}
\end{figure}

Some comments on the hypotheses (H0)-(H1)-(H2) are now in order.

\begin{remark} \label{rmk:commentary on assumptions}
A fundamental observation stemming directly from the definitions is that, in general, a stationary varifold may have multiple different tangent cones at a given point. In particular, (H0) and (H1) alone do not imply that $\var(T,Q)$ is the \emph{only} tangent cone to $V_0$ at $x_0$ (nor that all tangent cones to $V_0$ at $x_0$ are actually flat). As a matter of fact, the only conclusions that one may draw from (H0) and (H1) are that the density $\Theta_{V_0}(x_0)$ is an integer $Q \ge 2$ and that if $\bbC$ is tangent to $V_0$ at $x_0$ and $\spt(\|\bbC\|)$ is contained in an $n$-plane $T'$ then $\bbC=\var(T',Q)$ (as a consequence of the constancy lemma for stationary varifolds). The hypothesis (H2) resolves the ambiguity, so that in our setting $\var(T,Q)$ is the \emph{unique} tangent cone to $V_0$ at $x_0$. 
\end{remark}

The following is the main result of this paper.

{\samepage
\begin{theorem} \label{t:main}
If Assumption \ref{assumptions:complete} holds, then there exists a Brakke flow $\{V_t\}_{t \geq 0}$ with fixed boundary $\pa \Gamma_0$ such that:
\begin{itemize}
\item[(i)] $\lim_{t \to 0^+} \|V_t\| = \|V_0\| = \Ha^n \mres_{\Gamma_0}$;
\item[(ii)] $\|V_t\|(U)<\|V_0\|(U)$ for all $t>0$.
\end{itemize}
\end{theorem}
}

We observe explicitly that if $V_0$ is stationary then the constant flow $V_t = V_0$ for all $t \ge 0$ is an $n$-dimensional Brakke flow with fixed boundary $\pa \Gamma_0$: the only condition to verify is the validity of Brakke's inequality \eqref{brakineq}, which can be readily deduced by
\begin{eqnarray*}
\|V_t\|(\phi(\cdot,t))\Big|_{t=t_1}^{t_2} &=& \|V_0\|(\phi(\cdot,t_2)-\phi(\cdot,t_1)) \\
&=& \|V_0\|\left(\int_{t_1}^{t_2} \frac{\partial\phi}{\partial t}(\cdot,t) \, dt\right) \\
&\leq & \int_{t_1}^{t_2} \|V_0\|(\frac{\partial\phi}{\partial t}(\cdot, t)) \, dt\,. 
\end{eqnarray*}

Hence, for an initial datum as in Assumption \ref{assumptions:complete}, Theorem \ref{t:main} is a statement of non-uniqueness of Brakke flow.

\begin{remark}
In the light of the above discussion, one may consider the subset of $\IV_n(U)$ consisting of those stationary varifolds $V_0$ such that the assignment $V_t=V_0$ for all $t \ge 0$ defines the \emph{only} Brakke flow starting with $V_0$. We will say that such a stationary varifold is {\it dynamically stable}. A natural question is whether dynamically stable stationary varifolds enjoy better regularity properties than what stationarity alone is able to guarantee. By Theorem \ref{t:main}, if $V_0=\var(\Gamma_0,1)$ with $\Gamma_0$ as in Assumption \ref{ass:st19} and if $V_0$ is dynamically stable then the set of points satisfying (H1) \emph{and} (H2) is empty. As anticipated in the introduction, it is an open question whether this implies that the set of points satisfying (H1) alone is empty, too (in other words, whether (H0) and (H1) imply (H2)), although that appears to be the case in all known examples. Notice that if $V_0$ is such that the set of points as in (H1) is empty then, as a consequence of the celebrated regularity theorem by Allard \cite{Allard} and of the recent results by Naber and Valtorta \cite{NV_varifolds}, $\Gamma_0$ is an embedded real analytic $n$-dimensional minimal hypersurface in $U$ outside of a countably $(n-1)$-rectifiable singular set $\cS$ (hence, in particular, such that $\Ha^{n-1+\delta}(\cS)=0$ for every $\delta > 0$).
\end{remark}

As noted in the Introduction, the following is an immediate corollary of Theorem \ref{t:main} and the work \cite{MW21} by Minter and Wickramasekera. We refer to the Introduction and to \cite{MW21} for the definitions of stability of ${\rm Reg}(V_0)$ and of classical singularities of $V_0$.

\begin{corollary}\label{t:Neshan}
Let $U$, $\Gamma_0$, and $\{E_{0,i}\}_{i=1}^N$ satisfy Assumption \ref{ass:st19}, and let $V_0=\var(\Gamma_0,1)$. Suppose that $V_0$ satisfies (H0) and (H1), that the regular part ${\rm Reg}(V_0)$ is stable, and that in a neighborhood of $x_0$ the set of classical singularities $y$ of $V_0$ with $\Theta_{V_0}(y) < Q$ is empty. Then, there exists a non-trivial Brakke flow $\{V_t\}_{t \ge 0}$ starting with $V_0$, and satisfying the conclusions of Theorem \ref{t:main}. In particular, this applies if
\begin{itemize}
    \item[(a)] $\Gamma_0$ has the structure of rectifiable current, denoted $\llbracket \Gamma_0 \rrbracket$, $x_0 \in U \setminus \spt\|\partial\llbracket\Gamma_0\rrbracket\|$, and $Q=2$,
    \item[(b)] or $\Gamma_0$ has the structure of rectifiable current, $\llbracket\Gamma_0\rrbracket$ is area minimizing ${\rm mod}(2Q)$, and $x_0 \in U \setminus \spt\|\partial^p\llbracket\Gamma_0\rrbracket\|$.
\end{itemize}
\end{corollary}

The rest of the paper is devoted to the proof of Theorem \ref{t:main}. The proof is constructive, and it roughly proceeds as follows. First, we modify the set $\Gamma_0$ in a small ball of radius $\eps$ centered at $x_0=0$, so to have a quantifiable drop of its mass: we shall call this modification a ``hole nucleation'' in $\Gamma_0$. Then, we use our existence theorem from \cite{ST19} to produce a Brakke flow (with fixed boundary $\pa \Gamma_0$) starting with this modified set $\Gamma_0^\eps$. Using an iterative procedure which hinges upon Brakke's ``expanding hole lemma'' \cite[Lemma 6.5]{Brakke} (of which we present a detailed proof for the reader's convenience), we show that this hole ``expands'' at future times: in particular, a hole of size $\sim \eps$ at time $t=0$ becomes \emph{almost} a hole of size $\sim 2^j\,\eps$ at time $t \sim 2^{2j}\,\eps^2$. The ``\emph{almost}'' above accounts for an error occurring at each iteration which needs to be estimated in order to make sure that the evolving varifolds never re-gain the initial mass drop. The growth assumption (H2) is crucial to perform this estimate. The final product of this iterative process is a Brakke flow (depending on the size $\eps$ of the initial hole nucleation) which, at a time $\bar{t}>0$, has strictly less mass than $\Gamma_0$, with both $\bar{t}$ and the mass loss \emph{independent} of $\eps$: the Brakke flow in the statement is then obtained in the limit as $\eps \to 0^+$.

\section{Hole nucleation} \label{sec:hole}

In this section we show that, given $U$, $\Gamma_0$, and $\{E_{0,i}\}_{i=1}^N$ as in Assumption \ref{assumptions:complete}, there exist Brakke flows starting from a set $\Gamma_0^\eps$ obtained by modifying $\Gamma_0$ in a small ball $U_{2\eps}$ around $x_0=0$ in a way to obtain a quantifiable drop of its mass, and we discuss the limits of such Brakke flows as $\eps \to 0^+$. Informally, we may say that $\Gamma_0^\eps$ is obtained by ``making a hole'' of radius $\sim\eps$ in $\Gamma_0$. The details of the construction of $\Gamma_0^\eps$ are contained in the following lemma. 

\begin{lemma}\label{exbrakke}
Let $U$, $\Gamma_0$, $\{E_{0,i}\}_{i=1}^N$, and $T$ be as in Assumption \ref{assumptions:complete}. There exists $\eps_0>0$ such that, for all $\eps\in(0,\eps_0]$, there exist a relatively closed and
$\Ha^n$-rectifiable set 
$\Gamma_0^\eps\subset U$, a family $\{E_{0,i}^\eps\}_{i=1}^N$ of pairwise disjoint
non-empty open subsets of $U$ with finite perimeter such that:
\begin{itemize}
\item[(1)] $\Gamma_0^\eps\setminus U_{2\eps}=\Gamma_0\setminus U_{2\eps}$ 
and $E_{0,i}^\eps\setminus U_{2\eps}=E_{0,i}\setminus U_{2\eps}$ for each $i=1,\ldots,N$;
\item[(2)] $\Gamma_0^\eps=U\setminus \cup_{i=1}^N E_{0,i}^\eps$;
\item[(3)] $\Gamma_0^\eps\cap U_{2\eps}\subset \{(x',x_{n+1}) : |x_{n+1}|\leq G(x')\}$;
\item[(4)] $\Ha^n(\Gamma_0^\eps\cap U_{2\eps})\leq (4\eps)^n \omega_n(Q+1)$;
\item[(5)] $\Ha^n(\rC(T,0,\eps)\cap U_{2\eps}\cap \Gamma_0^\eps)\leq \omega_n \eps^n$.
\end{itemize}
\end{lemma}

\begin{proof}
As usual, we assume without loss of generality that $T=\R^n \times \{0\}$, and thus we write $x=(x',x_{n+1}) \in \R^{n+1}=T \oplus T^\perp$. Furthermore, we let $V_0=\var(\Gamma_0,1)$, so that, setting $\eta_r(x) = x/r$, we have $\lim_{r\rightarrow 0+}(\eta_r)_{\sharp}
V_0={\bf var}(T,Q)$ as varifolds. Define $\Gamma^\eps:=\eta_{\eps}(\Gamma_0)$, $V^{\varepsilon}:=(\eta_{\varepsilon})_\sharp
V_0$ and $E^\varepsilon_i:=\eta_{\varepsilon} (E_{0,i})$ for each $i=1,\ldots,N$ for simplicity. By (H2), there exists a sufficiently small
$\varepsilon_0>0$ such that for all $\varepsilon\in (0,\varepsilon_0]$, we have
\begin{equation}\label{heightbd}
U_{2}\cap \Gamma^\eps\subset \{|x_{n+1}|\leq 1/20\}.
\end{equation}
We may additionally assume that $E_i^{\varepsilon}\setminus U_2\neq \emptyset$ for all $i$. 
In the following, let $\delta =1/5$ and define the following function $\mathrm{g} \colon \R^{n+1} \to \R^{n+1}$ as in \cite[Proof of Lemma 4.7]{KimTone2}. For $\abs{x'} \leq 1$, we set
\begin{equation} \label{squash_cylinder}
{\rm g} (x',x_{n+1}) = 
\begin{cases}
(x',x_{n+1}) &\mbox{if $\abs{x_{n+1}} \geq \delta$}\,,\\
(x',0) &\mbox{if $\abs{x_{n+1}} \leq \frac{\delta}{2}$}\,,\\
(x',2\,x_{n+1}-\delta) &\mbox{if $\frac{\delta}{2} \leq x_{n+1} \leq \delta$}\,,\\
(x',2\,x_{n+1}+\delta) &\mbox{if $-\delta \leq x_{n+1}\leq - \frac{\delta}{2}$}\,,
\end{cases}
\end{equation}
whereas in the region $1 \leq \abs{x'} \leq 1+\delta$ we set
\begin{equation} \label{squash_annulus}
{\rm g} (x',x_{n+1}) = 
\begin{cases}
(x',x_{n+1}) &\mbox{if $\abs{x_{n+1}} \geq \delta$ or $\abs{x_{n+1}} \leq \abs{x'}-1$}\,,\\
(x',\abs{x'}-1) &\mbox{if $\abs{x'}-1\leq x_{n+1} \leq \frac{\abs{x'}-1}{2} + \frac{\delta}{2}$}\,,\\
(x',2\,x_{n+1}-\delta) &\mbox{if $\frac{\abs{x'}-1}{2} + \frac{\delta}{2} \leq x_{n+1} \leq \delta$}\,,\\
(x',1-\abs{x'}) &\mbox{if $\frac{1-\abs{x'}}{2} - \frac{\delta}{2} \leq x_{n+1} \leq 1-\abs{x'}$}\,,\\
(x',2\,x_{n+1}+\delta) &\mbox{if $-\delta \leq x_{n+1}\leq \frac{1-\abs{x'}}{2} - \frac{\delta}{2}$}\,.
\end{cases}
\end{equation}
Finally, we set 
\begin{equation} \label{squash_identity}
{\rm g}(x',x_{n+1}) = (x',x_{n+1}) \qquad \mbox{if $\abs{x'} > 1+\delta$}\,;
\end{equation}
see Figure \ref{fig:squash}. 
{\begin{figure}
\includegraphics[scale=0.8]{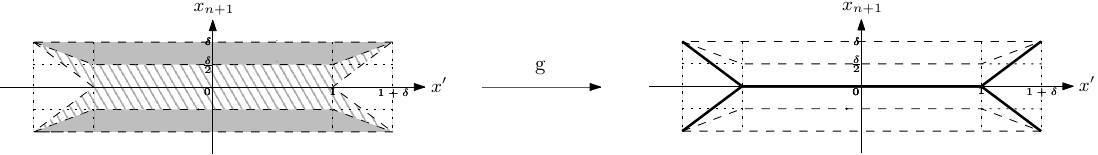}
\caption{\small{The map ${\rm g}$.}} \label{fig:squash}
\end{figure}}
One may check that ${\rm g}$ is a Lipschitz map with ${\rm Lip}({\rm g})\leq 2$. 
Next, define 
\begin{equation}\label{defG0}
\tilde E_i^\eps:
={\rm int}({\rm g}(E_i^\eps)) \quad \mbox{ and } \quad E_{0,i}^\eps:=\eta_{1/\eps}(\tilde E_i^\eps)
\end{equation} 
for each $i=1,\ldots,N$,
as well as
\begin{equation}\label{defG1}
\tilde\Gamma^\eps:=\eta_{\eps}(U)\setminus \cup_{i=1}^N\tilde E_i^\eps \quad\mbox{ and }\quad\Gamma_0^\eps:=\eta_{1/\eps}
(\tilde\Gamma^\eps)\,.
\end{equation}
Since ${\rm g}$ is a retraction map, one can check that 
$\tilde E_1^\eps,\ldots,\tilde E_N^\eps$ are mutually 
disjoint open sets, and so are $E_{0,1}^\eps,\ldots,E_{0,N}^\eps$. It follows from the definition of ${\rm g}$ that
$E_i^\eps\setminus U_2=\tilde E_i^\eps\setminus U_2$ for $i=1,\ldots,N$. Thus (1) is 
satisfied for $\Gamma_0^\eps$ and $E_{0,i}^\eps$, and (2) holds by construction. We next check 
\begin{equation}\label{defG2}
\tilde \Gamma^\eps\subset {\rm g}(\Gamma^\eps).
\end{equation}
We only need
to prove the inclusion on the set on which the map ${\rm g}$ is not one-to-one, namely on 
$\{(x',1-|x'|), (x',|x'|-1) : 1\leq |x'|< 1+\delta\}\cup (T\cap U_1)$. Note that 
any point $x$ of this set has the property that ${\rm g}^{-1}(x)$ is a closed line 
segment, say $I$, perpendicular 
to $T$. If $x\notin {\rm g}(\Gamma^\eps)$, that is, $I\cap \Gamma^\eps=I\cap \cup_{i=1}^N \partial E_i^\varepsilon=\emptyset$, then there must exist some $E_i^\eps$ 
such that $I\subset E_i^\eps$. Then one can see that $x$ is an interior point of 
${\rm g}(E_i^\eps)$, so that $x\in \tilde E_i^\eps$ and not in $\tilde\Gamma^\eps$. This proves \eqref{defG2}.
The inclusion \eqref{defG2} moreover proves that $\tilde\Gamma^\eps$ is countably $n$-rectifiable. 
From the 
definition of ${\rm g}$, we have $|T^{\perp}({\rm g}(x))|\leq |T^\perp(x)|$, and (3) follows from this fact.
Since ${\rm Lip}({\rm g})\leq 2$ and $\Ha^n(U_2\cap \Gamma^\eps)\leq 2^n\omega_n (Q+1)$ for 
all small $\eps$, the area formula guarantees that $\Ha^n(U_2\cap\tilde\Gamma^\eps)\leq 4^n\omega_n(Q+1)$.
This gives (4) after change of variables. 
In particular, $\tilde\Gamma^\eps$ has no interior points and finite $\Ha^n$ measure, it holds $\tilde\Gamma^\eps=\cup_{i=1}^N\pa \tilde E_i^\eps$, and the sets $E_{0,1}^\eps,\ldots,E_{0,N}^\eps$ have bounded perimeter.  
Finally, from the definition of ${\rm g}$, we have (writing $\rC(r)$ for $\rC(T,0,r)$)
\begin{equation}\label{defG4}
\begin{split}
&{\rm g}(\rC(1)\cap\{\delta/2\leq x_{n+1}\leq \delta\})=\rC(1)\cap \{0\leq x_{n+1}
\leq \delta\},\\
&{\rm g}(\rC(1)\cap\{-\delta\leq x_{n+1}\leq-\delta/2\})=\rC(1)\cap \{-\delta\leq x_{n+1}
\leq 0\}\,.
\end{split}
\end{equation}
Since $\cup_{i=1}^N \partial E_i^\varepsilon=\Gamma^\eps$ in $U_2$ and 
$E_1^{\eps},\ldots,E_N^\eps$ are mutually disjoint, \eqref{heightbd} shows that
there exist $i_1,i_2\in\{1,\ldots,N\}$ such that 
\begin{equation}
\label{defG3}
(U_{2}\cap\{x_{n+1}\geq \delta/2\})\subset E_{i_1}^\varepsilon\,\,\mbox{ and }\,\,
(U_{2}\cap\{x_{n+1}\leq -\delta/2\})\subset E_{i_2}^\varepsilon.
\end{equation}
We claim that
\begin{equation}\label{defG5}
\rC(1)\cap U_{2}\cap \tilde\Gamma^\eps\subset T\cap B_{1}.
\end{equation}
Note that \eqref{defG4} and \eqref{defG3} imply that $\rC(1)\cap
\{0<x_{n+1}\leq \delta\}\subset \tilde E_{i_1}^\eps$ and $\rC(1)\cap
\{-\delta\leq x_{n+1}<0\}\subset \tilde E_{i_2}^\eps$. Thus, we have $\tilde\Gamma^\eps\cap \rC(1)\cap \{-\delta
\leq x_{n+1}\leq \delta\}\subset T\cap B_1$. Since $U_2\cap \Gamma^\eps\cap \{|x_{n+1}|\geq \delta\}=\emptyset$, this proves \eqref{defG5}, and consequently we have (5). 
\end{proof}

By \cite[Theorems 2.2 \& 2.3]{ST19} we have then immediately the following existence of Brakke flows starting with $\Gamma_0^\eps$. 

\begin{proposition} \label{prop:brakke_eps}
With $\Gamma_0^\eps$ and $\{E_{0,i}^\eps\}_{i=1}^N$ given in Lemma \ref{exbrakke}, there
exists a Brakke flow $\{V_t^\eps\}_{t\geq 0}$ with fixed boundary $\pa\Gamma_0$ and 
$\|V_0^\eps\|=\Ha^n\mres_{\Gamma_0^\eps}$. For each $i=1,\ldots,N$, there exists a one-parameter
family $\{E_i^\eps(t)\}_{t\geq 0}$ of open sets $E_i^\eps(t)\subset U$ with the properties described in
\cite[Theorem 2.3]{ST19}. 
\end{proposition}
\begin{proof}
To apply \cite[Theorem 2.2 \& 2.3]{ST19} we only need to check that (A1)-(A4) in Assumption \ref{ass:st19} are satisfied with $U$, $\Gamma_0^\eps$, and $\{E_{0,i}^\eps\}_{i=1}^N$ for each $\eps\in(0,\eps_0]$. These follow from Lemma \ref{exbrakke} (1),(2).
\end{proof}

\begin{proposition} \label{prop:convergence}
For any sequence $\{\eps_j\}_{j=1}^\infty\subset(0,\eps_0]$ converging to $0$, there exist a subsequence
(denoted by the same index) and a Brakke flow $\{V_t\}_{t\geq 0}$ with fixed boundary $\pa\Gamma_0$ 
such that $\lim_{j\rightarrow\infty} \|V_t^{\eps_j}\|=\|V_t\|$ in $U$ for each $t\geq 0$ and 
$\lim_{t\rightarrow 0+}\|V_t\|=\|V_0\|=\Ha^n\mres_{\Gamma_0}$. 
\end{proposition}

\begin{proof} 
Since $\|V_t^\eps\|(U)\leq \|V_0^\eps\|(U)=\Ha^n(\Gamma_0^\eps)
= \Ha^n(\Gamma_0)+o(1)$ as $\eps\rightarrow 0^+$ for all $t> 0$, we have a uniform mass bound for 
the family. Such a family of Brakke flows is known to be compact (see \cite{Ilm1} and \cite[Section 3.2]{Ton1}),
thus there exists a subsequence (denoted by the same index) and a limit Brakke flow $\{V_t\}_{t\geq 0}$ such that 
$\lim_{j\rightarrow\infty}\|V_t^{\eps_j}\|=\|V_t\|$ as Radon measures on $U$ for all $t\geq 0$. By Lemma \ref{exbrakke}(4),
we also have $\|V_0\|=\Ha^n\mres_{\Gamma_0}$. 
For each $B_r(x)\times[t_1,t_2]\ssubset U\times(0,\infty)$ and $i\in\{1,\ldots,N\}$, the argument for the proof of 
\cite[Theorem 3.5(6)]{KimTone} (which is equally valid for the fixed boundary case away from $\pa U$) 
shows that $\mathcal L^{n+1}(E_i^\eps (t)\cap B_r(x))$
is $\frac12$-H\"{o}lder continuous as a function of $t\in[t_1,t_2]$, with uniformly bounded H\"{o}lder
norm independently of $\eps$. Also, if $0\notin B_{r+2\eps} (x)$, the same proof
shows that $\mathcal L^{n+1}(B_{r}(x)\cap E_i^\eps(t))$ is continuous at $t=0$ with uniform
modulus of continuity with respect to $\eps$. 
Since $\|\pa^*E_i^{\eps}(t)\|\leq \|V_t^{\eps}\|$ 
for all $t\geq 0$ (see \cite[Theorem 2.3(8)]{ST19}), and since the latter is uniformly bounded, 
by a suitable diagonal argument and the
uniform continuity in $t$, one can prove that there exists a further subsequence (denoted by the
same index) and a family of sets of finite perimeter $\{\tilde E_{i}(t)\}_{t\geq 0}$ in $U$ 
for $i=1,\ldots,N$
such that $\lim_{j\rightarrow\infty} \mathcal L^{n+1}(E_i^{\eps_j}(t)\,\triangle\, \tilde E_i(t))=0$ 
for all $t\geq 0$ and $\mathcal L^{n+1}(B_r(x)\cap \tilde E_i(t))$ is $C^{1/2}((0,\infty))\cap C([0,\infty))$
as a function of $t$ for any $B_r(x)\ssubset U$. For each $t\geq 0$, $\{\tilde E_i(t)\}_{i=1}^N$ 
satisfies $\mathcal L^{n+1}(\tilde E_i(t)\cap\tilde E_{i'}(t))=0$ for $i\neq i'$ and $\mathcal L^{n+1}(U\setminus\cup_{i=1}^N
\tilde E_i(t))=0$. By Lemma \ref{exbrakke}(1), we also have $\tilde E_i(0)=
E_{0,i}$ for each $i=1,\ldots,N$. 

Define the space-time
Radon measure $d\mu:=d\|V_t\|dt$ on $U\times(0,\infty)$ and 
$({\rm spt}\,\mu)_t:=\{x\in U : (x,t)\in {\rm 
spt}\,\mu\}$ for each $t>0$. By the fact that $\{V_t\}_{t\geq 0}$ is a Brakke flow, we have for all 
$t>0$ 
\begin{equation}\label{fmeas}
\mbox{${\rm spt}\,\|V_t\|\subset ({\rm spt}\,\mu)_t$ and 
$\Ha^n(({\rm spt}\,\mu)_t\cap \tilde U)<\infty$ for all $\tilde U\ssubset U$}
\end{equation} 
by \cite[Lemma 10.1]{KimTone} and \cite[Corollary
10.8]{KimTone}, respectively. Since 
\begin{equation}\label{fmeas2}
\|\pa^*\tilde E_i(t)\|\leq \liminf_{j\rightarrow\infty} \|\pa^* E_i^{\eps_j}
(t)\|\leq \lim_{j\rightarrow\infty}\|V_t^{\eps_j}\|= \|V_t\|
\end{equation}
for all $t\geq 0$, \eqref{fmeas} shows that ${\rm spt}\,\|\pa^*
\tilde E_i(t)\|\subset({\rm spt}\,\mu)_t$ for all $t>0$. In particular, on each 
connected component of $U\setminus ({\rm spt}\,\mu)_t$, $\chi_{\tilde E_i(t)}$ is constant. Then, it 
follows using the continuity property of $\mathcal L^{n+1}(B_r(x)\cap \tilde E_i(t))$ that the open set $(U\times(0,\infty))\setminus{\rm spt}\,\mu$ 
can be decomposed into mutually disjoint open sets $E_1,\ldots,E_N$ such that $\cup_{i=1}^N E_i=
(U\times(0,\infty))\setminus{\rm spt}\,\mu$ and such that $\mathcal L^{n+1}(\tilde E_i(t) \,\triangle \, \{x\in U: (x,t)\in E_i\})=0$ for all $t>0$. We may redefine $E_i(t)=\{x\in U:(x,t)\in E_i\}$, which is open and
$\mathcal L^{n+1}(\tilde E_i(t)\,\triangle\, E_i(t))=0$.
By definition, we have $U\setminus \cup_{i=1}^N E_i(t)=({\rm spt}\,\mu)_t$ and \eqref{fmeas}
shows that $U\setminus\cup_{i=1}^N E_i(t)$ has no interior points, so that 
we have $\cup_{i=1}^N \pa E_i(t)=({\rm spt}\,\mu)_t$ for all $t>0$. The continuity at $t=0$ 
shows that $\lim_{t\rightarrow 0+}\mathcal L^{n+1}(E_i(t)\,\triangle\, E_{0,i})=0$ for $i=1,\ldots,N$.

By \cite[Theorem 2.3(5)(11)]{ST19}, for each $j$ and
all $t\geq 0$,
${\rm spt}\,\|V_t^{\eps_j}\|\subset {\rm conv}(\Gamma_0^{\eps_j}\cup\pa\Gamma_0)$. Since 
the difference of $\Gamma_0^{\eps_j}$ and $\Gamma_0$ lies within $U_{2\eps_j}$, we may conclude that
${\rm spt}\|V_t\|\subset{\rm conv}(\Gamma_0\cup\pa\Gamma_0)$ for all $t\geq 0$, and one may deduce that
$({\rm spt}\,\mu)_t\subset {\rm conv}(\Gamma_0\cup\pa\Gamma_0)$ for $t>0$. We also can see from the
last claim that 
\begin{equation}\label{fmeas3}
E_i(t)\setminus{\rm conv}(\Gamma_0\cup\pa\Gamma_0)=E_{0,i}\setminus{\rm conv}(
\Gamma_0\cup\pa\Gamma_0)
\end{equation} 
for all $t\geq 0$ and $i=1,\ldots,N$. 

Next, we prove that
$V_t$ has a fixed boundary $\pa\Gamma_0$, i.e., $({\rm clos}\,({\rm spt}\|V_t\|))\setminus U=\pa\Gamma_0$.
The inclusion $\subset$ follows from ${\rm spt}\|V_t\|\subset{\rm conv}(\Gamma_0\cup\pa\Gamma_0)$,
the definition of $\pa\Gamma_0$ and the strict convexity of $U$.
For the converse inclusion, assume that we have $x\in\pa\Gamma_0$ and there exists $r>0$ such that
${\rm spt}\,\|V_t\|\cap U_r(x)=\emptyset$. Then we have $\|V_t\|(U_r(x))=0$. But then \eqref{fmeas2} shows $\|\pa^*E_i(t)\|(U_r(x)\cap U)=0$ for $i=1,\ldots,N$. On the other hand, by \eqref{fmeas3} and (A4) of Assumption \ref{ass:st19}, we must have
some $i_1\neq i_2$ such that $U_r(x)\cap E_{i_k}(t)\neq\emptyset$ for $k=1,2$. These are not 
compatible. Thus we have $({\rm clos}\,({\rm spt}\|V_t\|))\setminus U=\pa\Gamma_0$.

Finally, $\lim_{t\rightarrow 0+} \|V_t\|=\Ha^n\mres_{\Gamma_0}$ follows from the argument in
\cite[Proposition 6.10]{ST19} under the condition (A5) of Assumption \ref{ass:st19} that $\Ha^n(\Gamma_0\setminus \cup_{i=1}^N \pa^* E_{0,i})
=0$. Indeed, for any $\phi\in C_c(U;\mathbb R^+)$, $$\limsup_{t\rightarrow 0+}\|V_t\|(\phi)
\leq \|V_0\|(\phi)= \Ha^n\mres_{\Gamma_0}(\phi)$$ follows from the property
of Brakke flow, and
\begin{equation}\label{fmeas4}
\begin{split}
&2\Ha^n\mres_{(\cup_{i=1}^N\pa^* E_{0,i})}(\phi) =\sum_{i=1}^N\|\nabla\chi_{E_{0,i}}\|(\phi)\leq
\sum_{i=1}^N \liminf_{t\rightarrow 0+}\|\nabla \chi_{E_i(t)}\|(\phi) \\ &
\leq \liminf_{t\rightarrow 0+}\sum_{i=1}^N \|\nabla\chi_{E_i(t)}\|(\phi)\leq 2\liminf_{t\rightarrow 0+}\|V_t\|(\phi). 
\end{split}
\end{equation}
These show that $\lim_{t\rightarrow 0+}\|V_t\|(\phi)=\Ha^n\mres_{\Gamma_0}(\phi)$ if $\Ha^n(\Gamma_0\setminus \cup_{i=1}^N \pa^* E_{0,i})
=0$. 
\end{proof}

\section{Brakke's expanding holes lemma} \label{sec:expanding holes lemma}

In this section, we discuss Brakke's expanding holes lemma \cite[Lemma 6.5]{Brakke}, which is a key tool towards the proof of Theorem \ref{t:main}. Given its importance in the following arguments, and for the reader's convenience, we provide a detailed proof. The lemma is valid for Brakke flow of any codimension,
and we will state it and prove it in such generality. Hence, in this section $k$ will be a fixed integer in $\{1,\ldots,n\}$, and $T$ will be a plane in $\bG(n+1,k)$. Before stating the lemma, we will need some preliminary notation. 
\begin{definition}\label{chi}
Let $\chi \colon \R^{k} \to \R^+$ be a smooth cut-off function $0 \le \chi \le 1$, such that:
\begin{itemize} 
\item[(a)] $\chi(x)$ is a decreasing function of the radial variable $r = \abs{x}\,;$
\item[(b)] $\spt(\chi) \subset U^k_1(0)\,;$
\item[(c)] $\chi(x) = 1$ if $0 \leq \abs{x} < 1 - \zeta$ for a small positive number $\zeta$\,.
\end{itemize}
\end{definition}

We will denote
\begin{equation} \label{rho!!!}
\rho := \sup_{x \in \R^{k}} \left\lbrace  \abs{\nabla\chi(x)} + 2\|D^2\chi(x)\| \right\rbrace\,,
\end{equation}
and we shall often use the fact that $\abs{\nabla\chi(x)}^2/\chi(x)\leq 2\sup\|D^2
\chi\|\leq \rho$. 

For a radius $R > 0$, we set 
\begin{equation} \label{cylindrical cut off}
\chi_{T,R}(x) := \chi (T(x)/R)\,.
\end{equation}
Since the plane $T$ will always be kept fixed, we will drop the subscript $_T$ in \eqref{cylindrical cut off} and denote the cylindrical cut-off at scale $R$ simply by $\chi_R$. Along the same lines, we also recall that $\rC(R)$ denotes the infinite cylinder orthogonal to $T$ centered at the origin and with radius $R$.

Let us also collect the following well-known facts concerning the orthogonal projection operators onto planes in $\bG(n+1,k)$. The reader can consult \cite[Lemma 11.1]{Kasai-Tone} for their proofs.

\begin{lemma}
For $S,T \in \bG(n+1,k)$ and $v \in \R^{n+1}$, the following holds.
\begin{align}
{\rm I} \cdot T &= k\,, \quad T^t=T\,, \quad T \circ T = T\,, \quad T \circ T^\perp=0\,. \label{grass1} \\
0 & \leq k - S \cdot T = S^\perp \cdot T \leq k \, \|S-T\|^2\,. \label{grass2} \\
0 & \leq \|S-T\|^2 \leq (S-T) \cdot (S-T) = 2\, T^\perp \cdot S\,. \label{grass3} \\
\abs{T(S^\perp(v))} & \leq \|T-S\|\,\abs{v}\,. \label{grass4} \\
\abs{T(S^\perp(T(v)))} & \leq \|T-S\|^2\,\abs{v}\,. \label{grass5}
\end{align}
\end{lemma}

Finally, the following lemma estimates the tilt of tangent planes to an integral varifold with respect to a reference plane $T$ in terms of the $L^2$-excess of $V$ with respect to $T$ and the $L^2$ norm of the generalized mean curvature vector. The proof can be found in \cite[Lemma 11.2]{Kasai-Tone}.

\begin{lemma} 
    Let $U \subset \R^{n+1}$ be open, suppose that $V \in \IV_k(U)$ admits generalized mean curvature vector $h(\cdot,V)$, and let $T \in \mathbf{G}(n+1,k)$ and $\phi \in C^1_c (U;\R^+)$. Then 
    \begin{equation} \label{e:tilt of tg}
        \begin{split}
             &\int_{\mathbf{G}_k(U)} \|S-T\|^2\,\phi^2(x) \, dV(x,S) \\
             &  \qquad \qquad \qquad \leq 4\,\left( \int_U |h(x,V)|^2\,\phi^2(x)\,d\|V\|(x)\right)^{\frac12}\left( \int_U |T^\perp(x)|^2 \, \phi^2(x) \, d\|V\|(x)\right)^{\frac12} \\
             & \qquad \qquad \qquad + 16 \, \int_U |T^\perp(x)|^2 \, |\nabla \phi(x)|^2 \, d\|V\|(x)\,.
        \end{split}
    \end{equation}
\end{lemma}

\begin{remark}
   In the following we will need to apply a slightly modified version of the estimate \eqref{e:tilt of tg}. Precisely, if $\phi \in C^2_c(U;\R^+)$ is such that $\{\phi = 0\} \subset \{\nabla\phi = 0\}$ then 
   \begin{equation} \label{e:tilt of tg mod}
        \begin{split}
             &\int_{\mathbf{G}_k(U)} \|S-T\|^2\,|\nabla\phi(x)|^2 \, dV(x,S) \\
             &  \qquad   \leq 4\,\left( \int_U |h(x,V)|^2\,\phi^2(x)\,d\|V\|(x)\right)^{\frac12}\left( \int_U |T^\perp(x)|^2 \, |\nabla\phi(x)|^4\,\phi^{-2}(x) \, d\|V\|(x)\right)^{\frac12} \\
             & \qquad   + 16 \, \int_U |T^\perp(x)|^2 \, |\nabla |\nabla\phi(x)| |^2 \, d\|V\|(x)\,.
        \end{split}
    \end{equation}
The proof can be obtained by repeating \emph{verbatim} the proof of \cite[Lemma 11.2]{Kasai-Tone} with $\phi$ replaced by $|\nabla \phi|$ until the last inequality of formula (11.9): there, first multiply and divide by $\phi$ (on the set where $\nabla\phi \neq 0$) and then apply the Cauchy-Schwarz inequality to deduce \eqref{e:tilt of tg mod}.
\end{remark}

\begin{lemma}[Brakke's expanding holes lemma] \label{l:exp_holes}
Let $T \in \bG(n+1,k)$, $0 \leq t_1 < t_2 < \infty$, $0 < R_1 < R_2$, $0<\hat R_1<\hat R_2$ and set
\[
\sigma := \frac{R_2^2 - R_1^2}{t_2 - t_1}\,, \qquad R(t)^2 := R_1^2 + \sigma\,(t - t_1)\,, \qquad \phi_t(x) := \chi_{R(t)}(x) = \chi(T(x)/R(t))\,.
\]
Let $\{V_t\}_{t \in[t_1,t_2]}$ be a $k$-dimensional Brakke flow in $C(R_2)\cap \{|T^{\perp}(x)|<\hat R_2\}$ such that 
\begin{equation}
\label{sp-bd}
{\rm spt}\,\|V_t\|\cap \{\hat R_1<|T^\perp(x)|<\hat R_2\}=\emptyset \qquad \mbox{for all $t\in [t_1,t_2]$}.
\end{equation}
For a.e.~$t \in \left[t_1,t_2\right]$, define the functions $\alpha(t)$ and $\mu(t)$ by
\begin{align} \label{small_height}
\mu(t)^2 &:= \int_{\bC(R(t))} \abs{T^\perp(x)}^2 \, d\|V_t\|(x)\,, \\ \label{d:L2_h}
 \alpha(t)^2 &:= \int \abs{h(x,V_t)}^2 \, \phi_t^2(x) \, d\|V_t\|(x)\,.
\end{align}
Then, we have for a.e.~$t\in[t_1,t_2]$
\begin{equation} \label{e:dissipation}
\delta (V_t, \phi_t^2) (h(\cdot, V_t)) < - \frac{\alpha(t)^2}{2} + 320\, \rho^2 R(t)^{-4}\mu(t)^2\,.
\end{equation}
Furthermore, there is $M = M(k,\sigma,\rho) < \infty$ such that if $\mu \in \left[0,\infty\right)$ satisfies
\begin{equation} \label{scale invariant bound}
\mu(t)^2 \le \mu^2\,R(t)^{k+2} \qquad \mbox{for a.e.~$t \in \left[t_1,t_2\right]$}\,,
\end{equation}
then 
\begin{equation} \label{e:almost_mass_monotonicity}
R_2^{-k} \|V_{t_2}\|(\phi_{t_2}^2) \leq R_1^{-k} \|V_{t_1}\| (\phi_{t_1}^2) + M \mu^2 \log(R_2/R_1)\,.
\end{equation} 
\end{lemma}

\begin{remark}
The main conclusion of the lemma, equation \eqref{e:almost_mass_monotonicity}, establishes an upper bound on the gain of mass density ratio for Brakke flow at times $t_1,t_2$ in enlarging cylinders $\rC(R_1) \subset \rC(R_2)$. The difference in mass density ratios is bounded above by the supremum, in the interval $[t_1,t_2]$, of the (scale invariant) $L^2$-excess of $V_t$ with respect to the plane $T$, namely the function $R(t)^{-(k+2)}\mu(t)^2$.
\end{remark}

\begin{proof}
In the following, we use test functions $\phi_t^2$ in \eqref{brakineq}, which do not have compact support
in $\rC(R_2)$. On the other hand,
due to \eqref{sp-bd}, we may multiply $\phi_t^2$ by a suitable cut-off function which is identically
equal to $1$ on $\{|T^\perp(x)|\leq \hat R_1\}$ and which vanishes on $\{|T^\perp(x)|\geq \hat R_2\}$
and so that the resulting functions belong to $C^\infty_c(\rC(R_2))$. 
The computation is not affected at all so that it is understood in the following that we implicitly modify
$\phi_t^2$ as such without changing the notation. 
First, we show the validity of the dissipation inequality \eqref{e:dissipation}. Applying \eqref{defFV5} with $\phi=\phi_t^2$, using that 
\begin{equation} \label{derivatives of phi}
\nabla [\phi_t^2] = 2\phi_t\nabla\phi_t\,, \qquad \nabla\phi_t(x) = R(t)^{-1} T[\nabla\chi(T(x)/R(t))] \in T\,,
\end{equation}
and exploiting Brakke's perpendicularity Theorem \ref{BPT} we calculate
\begin{equation} \label{diss1}
\begin{split}
\delta (V_t, \phi_t^2)(h(\cdot, V_t)) &= -\int \abs{h(x,V_t)}^2 \, \phi_t^2(x) \, d\|V_t\|(x) + 2 \int h(x,V_t) \cdot S^\perp(\nabla \phi_t(x)) \, \phi_t(x) \, dV_t(x,S)\\
&\leq - \alpha(t)^2 + 2 \int \abs{h(x,V_t)}\, \|S - T\|\, \abs{\nabla\phi_t(x)}\, \phi_t(x) \, dV_t(x,S) \\
&\leq -\frac{3}{4} \alpha(t)^2 + 4\int \|S - T \|^2 \, \abs{\nabla\phi_t(x)}^2 \, dV_t(x,S)\,.
\end{split}
\end{equation}
In order to estimate the second term, we apply \eqref{e:tilt of tg mod} to deduce 
\begin{equation} \label{diss2}
\begin{split}
&\quad\int \|S - T \|^2 \,  \abs{\nabla\phi_t(x)}^2 \, dV_t(x,S) \\ & \leq  16 \int \abs{T^\perp(x)}^2 \, \abs{\nabla \abs{\nabla\phi_t(x)}}^2 \, d\|V_t\|(x) \\
& \hspace{2cm}+ 4\left\lbrace \left(\int \abs{h(x,V_t)}^2 \, \phi_t^2(x) \, d\|V_t\|(x)\right) \left( \int \abs{T^\perp(x)}^2 \, \abs{\nabla\phi_t(x)}^4 \, \phi_t^{-2}(x)  \right)  \right\rbrace^{\sfrac{1}{2}} \\
&\leq 16 \rho^2 R(t)^{-4} \mu(t)^2 + 4 \rho R(t)^{-2} \alpha(t) \mu(t) \\
&\leq \frac{\alpha(t)^2}{16} + 80 \rho^2 R(t)^{-4} \mu(t)^2\,.
\end{split}
\end{equation}
Equations \eqref{diss1} and \eqref{diss2} together prove \eqref{e:dissipation}. \\

In order to prove \eqref{e:almost_mass_monotonicity}, we observe that, heuristically,
\[
\pa_t \left[ R(t)^{-k} \, \|V_t\|(\phi_t^2) \right] = -k R'(t) R(t)^{-k-1} \, \|V_t\|(\phi_t^2) + R(t)^{-k} \pa_t\left[ \|V_t\|(\phi_t^2) \right]\,,
\] 
which can be expressed rigorously in terms of the inequality
\begin{equation} \label{integral_form}
\left. R(t)^{-k} \|V_t\|(\phi_t^2) \right|_{t=t_1}^{t_2} \leq \int_{t_1}^{t_2}  R(t)^{-k} D \left[\|V_t\|(\phi_t^2) \right] - k R'(t) R(t)^{-k-1}\, \|V_t\|(\phi_t^2) \,dt\,,
\end{equation}
where $D\left[ \|V_t\|(\phi_t^2)  \right]$ is the distributional derivative of $t \in \left[t_2,t_2\right] \mapsto \|V_t\|(\phi_t^2)$. On the other hand, since $\{V_t\}_{t \ge 0}$ is a Brakke flow, Brakke's inequality \eqref{brakineq} implies that, in the sense of distributions, 
\begin{equation} \label{brakke's ineq}
D \left[ \|V_t\|(\phi_t^2) \right] \leq \delta (V_t, \phi_t^2) (h(\cdot, V_t)) + \|V_t\| (\partial_t [\phi_t^2])\,.
\end{equation} 
Since \eqref{e:dissipation} controls the first addendum, we only have to estimate the second one. We first compute explicitly the time derivative $\pa_t [\phi_t^2]$, namely
\begin{eqnarray*}
\pa_t [\phi_t^2] &=& 2\, \phi_t\, \pa_t\phi_t = -2\, \phi_t \, R(t)^{-2}R'(t) \, T(x) \cdot \nabla \chi(T(x)/R(t)) = \\
&=& -2\,\phi_t\, R(t)^{-2}R'(t)  \, x \cdot T\left[ \nabla\chi(T(x)/R(t)) \right]\\
&=& -2\,\phi_t\, R(t)^{-1}R'(t) \, x \cdot \nabla \phi_t(x)\,,
\end{eqnarray*}
where we have used that $T$ is an orthogonal projection operator and the expression for $\nabla\phi_t$ in \eqref{derivatives of phi}. Hence, we have
\begin{equation} \label{to_estimate}
\|V_t\|(\partial_t[\phi_t^2]) = R(t)^{-1}R'(t) \underbrace{\int\left\lbrace -2\,\phi_t(x) \, x \cdot \nabla\phi_t(x) \right\rbrace \, d\|V_t\|(x)}_{A}\,.
\end{equation}

Now, observe that, by the definition of first variation and generalized mean curvature,
\begin{eqnarray*}
- \int \phi_t^2(x)\, T(x) \cdot h(x,V_t) \, d\|V_t\|(x) &=& \delta V_t(\phi_t^2\, T(\cdot))\\
&=& \int \nabla[\phi_t^2\,T(x)] \cdot S \, dV_t(x,S) \\
&=& k\, \|V_t\|(\phi_t^2) + \int 2\, \phi_t \left[T(x) \otimes \nabla\phi_t\right] \cdot S \, dV_t(x,S) \\
&&\qquad \qquad \qquad+ \int \phi_t^2 \, (T \cdot S-k) \, dV_t(x,S)\,,
\end{eqnarray*}
and since $[u \otimes v] \cdot S = S(v) \cdot u = S(u) \cdot v$ for a symmetric $S$, we have that
\begin{eqnarray*}
&&\int \left\lbrace -2\, \phi_t \,\nabla \phi_t \cdot S(T(x)) \right\rbrace \, dV_t(x,S)\\ && \hspace{1cm} = k\, \|V_t\|(\phi_t^2) + \underbrace{\int \phi_t^2(x)\, T(x) \cdot h(x,V_t) \, d\|V_t\|(x)}_{I_1} + \underbrace{\int \phi_t^2 \, (T \cdot S-k) \, dV_t(x,S)}_{I_2}\,.
\end{eqnarray*}

On the other hand, it also holds
\[
\int \left\lbrace -2 \, \phi_t \, \nabla \phi_t \cdot S(T(x))\right\rbrace \, dV_t(x,S) = A - \underbrace{\int \left\lbrace - 2 \,\phi_t(x)\, \nabla \phi_t(x) \cdot [x - S(T(x))] \right\rbrace \, dV_t(x,S)}_{I_3}\,, 
\]
so that
\begin{equation} \label{key}
A =k\, \|V_t\|(\phi_t^2) + I_1 + I_2 + I_3\,,
\end{equation}
and we can estimate the three pieces one at a time.\\

In order to estimate the term $I_1$, we use Brakke's perpendicularity theorem to write
\[
I_1=\int \phi_t^2(x)\, S^\perp(T(x)) \cdot h(x,V_t)\, dV_t(x,S)\,,
\]
so that, using $\abs{T(x)} \le R(t)$ on $\spt(\phi_t)$ together with $2\,R(t)R'(t)=\sigma$, we get from \eqref{grass1} that
\begin{equation} \label{est5}
\begin{split}
\abs{I_1} & \leq \int \phi_t^2(x)\, \|S-T\| \, \abs{T(x)} \, \abs{h(x,V_t)} \, dV_t(x,S) \ \\
& \le\frac{R(t)}{R'(t)}\, \frac{\alpha(t)^2}{4} + \frac{\sigma}{2} \int \phi^2_t(x) \, \|S-T\|^2 \, dV_t(x,S)\,.
\end{split}
\end{equation}

Using \eqref{grass2}, we have, instead: 
\begin{equation} \label{est1}
\abs{I_2} \leq k \int \phi^2_t(x)\, \|S-T\|^2\, dV_t(x,S)\,.
\end{equation}

Because $\nabla\phi_t \in T$ and $T \circ T = T$, $\nabla \phi_t \cdot [x - S(T(x))] = \nabla \phi_t \cdot T(S^\perp(T(T(x))))$, and thus \eqref{grass5} and Young's inequality give
\begin{equation} \label{est2}
\begin{split}
\abs{I_3} & \le 2 \int \phi_t(x) \, \abs{\nabla \phi_t(x)}\, \|S-T\|^2\, \abs{T(x)} \, dV_t(x,S) \\
&\leq \int \phi^2_t(x) \, \|S-T\|^2\, dV_t(x,S) + R(t)^2\int \abs{\nabla \phi_t(x)}^2 \, \|S-T\|^2\, dV_t(x,S)\,.
\end{split}
\end{equation}

In turn, by \eqref{e:tilt of tg} we can further estimate

\begin{equation} \label{est3}
\int \phi^2_t(x) \, \|S-T\|^2\, dV_t(x,S) \leq  4\alpha(t)\mu(t) + 4 \rho R(t)^{-2} \mu^2(t)\,,
\end{equation}

so that plugging \eqref{diss2} and \eqref{est3} into \eqref{est5}, \eqref{est1}, and \eqref{est2} yields
\begin{equation} \label{global_estimate}
\abs{I_1} + \abs{I_2} + \abs{I_3} \leq \frac{R(t)}{R'(t)} \, \frac{\alpha(t)^2}{4} + C\, (\alpha(t)\mu(t) + R(t)^{-2} \mu(t)^2)\,,
\end{equation}
where $C$ is a constant depending only on $k,\sigma$ and $\rho$.

In particular, from \eqref{key} and the definition of $A$ we conclude the following bound:
\begin{equation} \label{almost_final}
\begin{split}
&\|V_t\|(\pa_t[\phi_t^2]) - k\, R(t)^{-1} R'(t) \, \|V_t\|(\phi_t^2) \\& \leq \frac{\alpha(t)^2}{4} + C\, R'(t)R(t)^{-1}\, (\alpha(t)\mu(t) + R(t)^{-2} \mu(t)^2)\\
&  \leq \frac{\alpha(t)^2}{2} + C \, \abs{R'(t)R(t)^{-1}}^2 \, \mu(t)^2 + C \, R'(t)R(t)^{-3}\,\mu(t)^2\\
&\leq \frac{\alpha(t)^2}{2}+CR(t)^{-4}\mu(t)^2\,.
\end{split}
\end{equation}
In the last line, we used the identity $\sigma = 2\,R'(t)R(t)$: the constants $C$ are different from line to line throughout the calculation, but they all depend only on $k,\sigma$ and $\rho$.
Now, we first use \eqref{brakke's ineq}, \eqref{e:dissipation} and \eqref{almost_final}, and then we multiply by $R(t)^{-k}$ in order to gain, thanks to $\sigma = 2\,R'(t)R(t)$ and the definition of $\mu$ in \eqref{scale invariant bound}, the estimate
\begin{equation} \label{integrand bound}
\begin{split}
R(t)^{-k} D \left[\|V_t\|(\phi_t^2) \right] &- k R'(t) R(t)^{-k-1}\, \|V_t\|(\phi_t^2) \\
&\leq \abs{R'(t)R(t)^{-1}} \, M \, \mu^2\,,
\end{split}
\end{equation} 
where $M$ is a constant depending only on $k,\sigma$ an $\rho$. 

The conclusion \eqref{e:almost_mass_monotonicity} then follows plugging \eqref{integrand bound} into \eqref{integral_form}.
\end{proof}

\section{$L^2$ excess estimates} \label{sec:excess}

This section contains the technical results which will be needed in the proof of Theorem \ref{t:main} in order to estimate the $L^2$ excess terms in the iterative applications of \eqref{e:almost_mass_monotonicity}, representing the possible gains of mass density ratio at each iteration. A careful estimate of these terms is crucial to show that the limiting Brakke flow is not trivial. \\

We begin with the following result, which is an adaptation of \cite[Proposition 6.5]{Kasai-Tone}. It states that the (scale invariant) $L^2$ excess of varifolds evolving according to Brakke flow in a given ball can be estimated \emph{uniformly in time} with the $L^2$ excess of the initial datum in a larger ball, with an error terms which decays to zero exponentially fast as the magnifying factor of the ball diverges to infinity, \emph{provided} said varifolds have uniformly bounded mass density ratio in such larger ball. 

\begin{proposition}\label{p:L2_height_bound}
Let $R >0$, $2 \le L < \infty$, and let $\{V_t\}_{0 \leq t \leq R^2}$ be a $k$-dimensional Brakke flow in $U_{LR}$. Then, for every $T \in \bG(n+1,k)$, and for all $t \in \left[0,R^2 \right]$ we have
\begin{equation} \label{e:L^2_height_bound}
\begin{split}
R^{-(k+2)} \int_{U_R} \abs{T^\perp(x)}^2\,d\|V_t\| &\leq e^{1/4}\,  R^{-(k+2)} \int_{U_{LR}} \abs{T^\perp(x)}^2 \, d\|V_0\|\\
&\quad + c(n,k)\, L^{k+2}\, \exp\left( - (L-1)^2/8\right)\,\sup_{t \in \left[0,R^2\right]}\frac{\|V_t\|(U_{LR})}{(LR)^k}\,.
\end{split}
\end{equation}
\end{proposition}

\begin{proof}

Without loss of generality, we can assume $R=1$. Let $\psi \in C^\infty_c(U_L)$ be a radially symmetric cut-off function with $0 \le \psi \le 1$, $\psi \equiv 1$ in $B_{L-1}$, and $\abs{\nabla \psi}, \|D^2\psi\| \le c(n)$. Using that $\{V_t\}_{0 \le t \le 1}$ is a Brakke flow, we test Brakke's inequality \eqref{brakineq} with
\begin{equation} \label{test function parabolic}
\phi(x,t) := \abs{T^\perp(x)}^2\, \psi(x) \, \vrho(x,t)\,,
\end{equation} 
where $\vrho(x,t):=\vrho_{(0,2)}(x,t)$ is the $k$-dimensional backward heat kernel
\begin{equation} \label{rho heat}
\vrho_{(y,s)}(x,t) := \frac{1}{(4\pi(s-t))^{k/2}}\,\exp\left( -  \frac{\abs{x-y}^2}{4(s-t)}  \right) \,,
\end{equation}
with $y=0$ and $s=2$
and thus we obtain, writing $\psi=\psi(x)$, $\vrho=\vrho(x,t)$, and $h=h(x,V_t)$, and for any $\tau \leq 1$,
\begin{align}
\left. \int \abs{T^\perp(x)}^2\,\psi \vrho \, d\|V_t\| \right|_{t=0}^{\tau} & \leq  \int_0^{\tau} \int \left\lbrace - \abs{h}^2\,\psi \vrho\, \abs{T^\perp(x)}^2 + h \cdot \nabla (\psi\vrho \abs{T^\perp(x)}^2) \right\rbrace \notag \\ 
&  \qquad \qquad +\psi\, \abs{T^\perp(x)}^2\, \frac{\pa \vrho}{\pa t} \, d\|V_t\|dt\,. \label{l2hb:1}
\end{align}

Using the perpendicularity of the mean curvature \eqref{BPT}, and consequently the fact that
\[
0 \le \left|  h - \frac{S^\perp(\nabla\vrho)}{\vrho}  \right|^2 = \abs{h}^2 - \frac{2}{\vrho} (h \cdot \nabla \vrho) + \frac{\abs{S^\perp(\nabla\vrho)}^2}{\vrho^2} \qquad \mbox{for $V_t$-a.e. $(x,S)$, $t \in \left[0,\tau\right]$}\,,
\]
we can estimate the integrand on the right-hand side of \eqref{l2hb:1} by
\[
\begin{split}
& \qquad -\abs{h}^2\,\psi\vrho\,\abs{T^\perp(x)}^2 + (h \cdot \nabla\vrho)\, \psi\,\abs{T^\perp(x)}^2 + \vrho\, h\cdot \nabla (\psi\,\abs{T^\perp(x)}^2) + \psi\,\abs{T^\perp(x)}^2\,\frac{\pa\vrho}{\pa t}\\
&\leq \quad - (h \cdot \nabla\vrho)\, \psi\,\abs{T^\perp(x)}^2 + \frac{\abs{S^\perp(\nabla\vrho)}^2}{\vrho} \,\psi\,\abs{T^\perp(x)}^2 + \vrho\, h\cdot \nabla (\psi\,\abs{T^\perp(x)}^2) + \psi\,\abs{T^\perp(x)}^2\,\frac{\pa\vrho}{\pa t}\,.
\end{split}
\]
On the other hand, we have by the definition of generalized mean curvature vector and the properties of Brakke flow that for a.e. $t \in \left( 0, \tau \right)$
\[
\begin{split}
& \qquad \int \left\lbrace  - (h \cdot \nabla\vrho)\, \psi\,\abs{T^\perp(x)}^2 + \vrho\, h\cdot \nabla (\psi\,\abs{T^\perp(x)}^2)  \right\rbrace \, d\|V_t\| \\
 &= \quad \int \left\lbrace \nabla (\psi\,\abs{T^\perp(x)}^2\,\nabla\vrho) \cdot S - \nabla (\vrho\,\nabla(\psi\,\abs{T^\perp(x)}^2)) \cdot S \right\rbrace \, dV_t(x,S) \\
 &= \quad \int\left\lbrace  (D^2\vrho \cdot S)\, \psi \, \abs{T^\perp(x)}^2 - \vrho\, D^2(\psi \, \abs{T^\perp(x)}^2) \cdot S \right\rbrace \, dV_t(x,S)\,.
\end{split}
\]

It is easy to see by direct calculation that, for any $S \in \bG(n+1,k)$
\begin{equation} \label{nice calculation heat kernel}
(D^2\vrho \cdot S) + \frac{\abs{S^\perp(\nabla\vrho)}^2}{\vrho} + \frac{\pa\vrho}{\pa t} \equiv 0\,.
\end{equation}

Hence, we conclude from \eqref{l2hb:1} that
\begin{equation} \label{l2hb:2}
\left. \int \abs{T^\perp(x)}^2\,\psi \vrho \, d\|V_t\| \right|_{t=0}^{\tau} \leq - \int_0^{\tau} \int \vrho \, D^2(\psi \, \abs{T^\perp(x)}^2) \cdot S \, dV_t(x,S)\,dt\,.
\end{equation}

Now, using that $S$ is symmetric we can directly compute
\begin{equation} \label{l2hb:3}
- D^2(\psi\,\abs{T^\perp(x)}^2)\cdot S = - (D^2\psi \cdot S)\, \abs{T^\perp(x)}^2 - 4\, (\nabla \psi \otimes T^\perp(x)) \cdot S - 2\,\psi \, (T^\perp \cdot S)\,.
\end{equation}

Notice that $T^\perp \cdot S \geq 0$ by \eqref{grass2}, and that \eqref{grass3} and \eqref{grass4} allow to estimate
\begin{equation} \label{l2hb:4}
4\,\abs{(\nabla\psi \otimes T^\perp(x)) \cdot S} \leq 4 \sqrt{2} \, \abs{\nabla\psi} \, \abs{T^\perp(x)}\, \sqrt{T^\perp \cdot S} \leq 2\,\psi\, (T^\perp \cdot S) + 4\, \frac{\abs{\nabla\psi}^2}{\psi} \, \abs{T^\perp(x)}^2\,,
\end{equation}

so that, using $\abs{\nabla\psi}^2 / \psi \leq c(n)$, \eqref{l2hb:2} yields

\begin{equation} \label{l2hb:5}
\begin{split}
\left. \int \abs{T^\perp(x)}^2\,\psi \vrho \, d\|V_t\| \right|_{t=0}^{\tau} & \leq c(n) \int_0^1 \int \vrho\, \| D^2\psi\|\, \abs{T^\perp(x)}^2 \, d\|V_t\|dt \\
&\leq c(n,k)\, L^2 \, \exp\left( - \frac{(L-1)^2}{8} \right) \,\int_0^{\tau} \|V_t\|(U_L)\, dt\,,
\end{split}
\end{equation}
where in the last inequality we have used that $\|D^2\psi\|=0$ in the complement of $A=U_L \setminus B_{L-1}$, that $\abs{T^\perp(x)} \leq L$ for $x\in U_L$, and that $\vrho(x,t) \leq c(k)\, \exp(-(L-1)^2 / 8)$ for $t \in \left[0,1 \right]$ and $x \in A$. Since $L \ge 2$, $\psi \equiv 1$ in $B_1$. Using furthermore that $\psi \leq \chi_{U_L}$, that $\vrho(x,0) \leq (8\pi)^{-k/2}$ everywhere, and that $\vrho(x,\tau) \ge (8\pi)^{-k/2}\,e^{-1/4}$ for $x \in B_1$ and $\tau \leq 1$, we obtain \eqref{e:L^2_height_bound} from \eqref{l2hb:5}.
\end{proof}

When $\{V_t\}$ is the Brakke flow $\{V_t^\eps\}$ of Proposition \ref{prop:brakke_eps}, the last term of \eqref{e:L^2_height_bound} can be controlled by the localized Huisken's monotonicity formula.
\begin{proposition}\label{denpro}
Let $\{V^\eps_t\}_{t\geq 0}$ be the Brakke flow obtained in Proposition \ref{prop:brakke_eps}. Then 
there exists a constant $E_0$ depending only on $4 R_0:={\rm dist}(0,\pa U), \Gamma_0$, $n$, and $Q$ such that
\begin{equation}
\label{den-up}
\sup_{\eps\in(0,\eps_0)}\left(\sup_{R\in[\eps,R_0] ,\,s\in[0,R_0^2]} \frac{\|V_s^\eps\|(U_R)}{\omega_n R^n}
\right)\leq E_0\,.
\end{equation}
\end{proposition}
\begin{proof}
Let $\psi\in C_c(U_{2R_0})$ be a radially symmetric function such that $0\leq \psi\leq 1$, $\psi=1$ on 
$B_{R_0}$ and $\|\psi\|_{C^2}\leq c(R_0)$. We use $\psi(x)\vrho_{(0,s+R^2)}(x,t)$ in \eqref{brakineq}
(with $k=n$) and proceed as in the proof of Proposition \ref{p:L2_height_bound}. Then we obtain for $s\in(0,R_0^2]$
and $R\in[\eps,R_0]$
\begin{equation} \label{e:Huisken}
\left.\int \vrho_{(0,s+R^2)}\,\psi\, d\|V_t^\eps\| \right|_{t=0}^{s}  \leq  c(R_0) \, \sup_{t \in \left[0,s\right]} \|V_t^\eps\|(U_{R_0})\leq c(R_0)\Ha^n(\Gamma_0^\eps),
\end{equation} 
where $c(R_0)$ is an another constant depending only on $R_0$ and the last inequality is 
due to \eqref{brakineq2} and $\|V_0^\eps\|=\Ha^n\mres_{\Gamma_0^\eps}$. Since $\Ha^n(\Gamma_0^\eps)$ is
close to $\Ha^n(\Gamma_0)$ for small $\eps$, the right-hand side of \eqref{e:Huisken} is uniformly bounded.
For the left-hand side, the evaluation of $t=s$ gives 
\begin{equation}
\begin{split}
\int \vrho_{(0,s+R^2)}(x,s)\,\psi(x)\,d\|V_s^\eps\|(x)&\geq \int_{U_R}\frac{\psi(x)}{(4\pi R^2)^{n/2}}\exp\Big(-\frac{|x|^2}{4R^2}
\Big)\,d\|V_s^\eps\|(x) \\
&\geq \frac{e^{-1/4}}{(4\pi)^{n/2}} R^{-n}\|V_s^\eps\|(U_R),
\end{split}
\end{equation}
where we used $\psi=1$ on $U_R\subset B_{R_0}$. The evaluation of $t=0$ may be estimated using Fubini's Theorem as
\begin{equation}\label{den-up2}
\int \vrho_{(0,s+R^2)}\psi\,d\|V_0^\eps\|\leq (4\pi(s+R^2))^{-n/2}
\int_0^1 f(\lambda)\,d\lambda
\end{equation}
with $f(\lambda):=\|V_0^\eps\|(
\{x\in U_{2R_0}: \exp\big(-\frac{|x|^2}{4(s+R^2)}\big)\geq \lambda\})$. 
We next evaluate $f(\lambda)$ depending on the value of $\lambda$ in 
\begin{itemize}
\item[(a)] $(0,\exp\big(-4R_0^2/4(s+R^2)\big))$, 
\item[(b)] $[\exp\big(-4R_0^2/4(s+R^2)\big),\exp\big(-4\eps^2/4(s+R^2)\big))$ and 
\item[(c)] $[\exp\big(-4\eps^2/4(s+R^2)\big),1)$. 
\end{itemize}
In the case of (a), one can see that 
$f(\lambda)=\|V_0^\eps\|(U_{2R_0})$. For (b),
we use the fact that $\Ha^n(\Gamma_0^\eps\cap B_r)\leq c(\Gamma_0)r^n$ for $r\in[2\eps,2R_0]$ 
(which follows from the monotonicity formula for $\Gamma_0$ and Lemma \ref{exbrakke}(1)(4)) and obtain
\begin{equation*}
\begin{split}
f(\lambda)&=\Ha^n(\Gamma_0^\eps
\cap B_{\sqrt{4(s+R^2)\log(1/\lambda)}}) \\
&\leq c(\Gamma_0)(4(s+R^2)\log(1/\lambda))^{n/2}. 
\end{split}
\end{equation*}
For (c), the set in question is included in $U_{2\eps}$, so that $f(\lambda)$ is bounded by 
$(4\eps)^n\omega_n(Q+1)$ due to Lemma \ref{exbrakke}(4). Combining these estimates, we have
\begin{equation}\label{den-up4}
\begin{split}
\int_0^1f(\lambda)\,d\lambda&\leq e^{-R_0^2/(s+R^2)}\|V_0^\eps\|(U_{2R_0})+c(\Gamma_0)(4(s+R^2))^{n/2}\int_0^1\log(1/\lambda)^{n/2}
\,d\lambda \\
&+(4\eps)^n\omega_n(Q+1).
\end{split}
\end{equation}
Since $(4\pi(s+R^2))^{-n/2}e^{-R_0^2/(s+R^2)}$ and $\|V_0^\eps\|(U_{2R_0})$ are bounded 
uniformly and $R\geq \eps$, 
\eqref{den-up2} and \eqref{den-up4} show that $\int \vrho_{(0,s+R^2)}\psi\,d\|V_0^\eps\|$ is 
bounded depending only on $R_0, \Gamma_0,n$ and $Q$. The estimates \eqref{e:Huisken}-\eqref{den-up4} now
show \eqref{den-up}. 
\end{proof}

\smallskip

With a simple geometric argument, we are allowed to replace balls with cylinders in \eqref{e:L^2_height_bound} if the initial datum is sufficiently flat. For the proof, we show that there is an ``empty spot'' just above and
below the origin for all sufficiently small scale. 

\begin{proposition} \label{es-h}
Let $\{V_t^\eps\}_{t\geq 0}$ be the Brakke flow obtained in Proposition \ref{prop:brakke_eps}. 
Then there exists $r_1=r_1(n,r_0,\alpha)$ such that, for $R\in[\eps,r_1)$, $\eps\in (0,\eps_0)$
and $t\in [0,4R^2]$, we have
\begin{equation}\label{es-h1}
\rC (\sqrt 2 R)\cap \{\sqrt 2 R\leq |T^{\perp}(x)|\leq 2 R\}\cap {\rm spt}\,\|V_t^\eps\|=\emptyset\,.
\end{equation}
Moreover, for $2\leq L<\infty$ with $2LR<r_0$, we have
\begin{equation}\label{es-h6}
\begin{split}
\int_{\bC (\sqrt 2 R)\cap\{|T^{\perp}(x)|<\sqrt 2 R\}}|T^{\perp}(x)|^2\,d\|V_t^\eps\|&
\leq e^{1/4} \int_{U_{2LR}}|T^{\perp}(x)|^2\,d\|V_0^\eps\| \\
&+c(n)(RL)^{n+2}\exp(-(L-1)^2/8)E_0\,,
\end{split}
\end{equation}
where $E_0$ is as in Proposition \ref{denpro}. 
\end{proposition}
\begin{proof}
Set $\delta_1:=(8n+2)/(\sqrt 2 -1)$ and fix a sufficiently small $r_1=r_1(n,\alpha,r_0)>0$ so that 
\begin{align}
&\frac{\delta_1}{\log^\alpha (1/(r_1\delta_1))}<1\,,\mbox{ and }\label{es-h2} \\
&(\sqrt 2+2\delta_1) r_1<r_0\,.\label{es-h3}
\end{align}
Assume $R<r_1$. 
Let $A(t)$ be the closed ball with center at $(x',x_{n+1})=(0,R(\sqrt 2+\sqrt{\delta_1^2-8n-2}))$ and the 
radius given by $\sqrt{(R\delta_1)^2-2nt}$. The radius is chosen so that $\pa A(t)$ is a MCF and
\begin{align}
&A(0)\subset \{|T(x)|\leq R\delta_1\}\cap \{R< x_{n+1}< r_0\}, \label{es-h4} \\
&\rC (\sqrt{2} R)\cap \{\sqrt 2 R\leq x_{n+1}\leq 2R\}\subset A(4R^2). \label{es-h5}
\end{align}
Indeed, the minimum of $T^{\perp}(A(0))$ satisfies
$$R\left(\sqrt 2+\sqrt{\delta_1^2-8n-2}-\delta_1\right)=R\left(\sqrt 2-\frac{8n+2}{\delta_1+\sqrt{\delta_1^2-8n-2}}\right)>R\left(\sqrt 2-\frac{8n+2}{\delta_1}\right)
=R$$
by the definition of $\delta_1$ and the maximum satisfies
$$
R(\sqrt 2+\sqrt{\delta_1^2-8n-2}+\delta_1)<R(\sqrt 2+2\delta_1)<r_0
$$
by \eqref{es-h3} and $R<r_1$. These show \eqref{es-h4}. 
One can check by calculation that \eqref{es-h5} holds as well. 
By \eqref{es-h2}, Lemma \ref{exbrakke}(1)(4) and \eqref{growth}, one can show that 
$$\Gamma_0^\eps\cap \{|T(x)|\leq R\delta_1\}\cap \{|x_{n+1}|<r_0\}\subset \{|x_{n+1}|<R\}$$
and thus, by \eqref{es-h4}, $A(0)\cap \Gamma_0^\eps=\emptyset$. By Brakke's sphere barrier to external
varifold lemma, see \cite[\S\;3.7]{Brakke} and \cite[Lemma 10.12]{KimTone}, one can conclude that
$A(t)\cap {\rm spt}\,\|V_t^\eps\|
=\emptyset$ for $t\in[0,(R\delta_1)^2/2n)$ and in particular for $t\in [0,4R^2]$. This combined with
\eqref{es-h5} shows \eqref{es-h1} for the case of $x_{n+1}>0$, and the case of $x_{n+1}<0$ is 
symmetric. Finally, we have
$$\bC (\sqrt 2 R)\cap\{|T^{\perp}(x)|<\sqrt 2 R\}\subset U_{2R}.$$
We can then deduce \eqref{es-h6} from \eqref{p:L2_height_bound} with $R$ there replaced by $2R$ and $k=n$, and thanks to \eqref{den-up}. 
\end{proof}

\section{Proof of Theorem \ref{t:main}} \label{sec:main proof}

We are now in the position of proving Theorem \ref{t:main}. 
We fix $U$, $\Gamma_0$, and $\{E_{0,i}\}_{i=1}^N$ so that Assumption \ref{assumptions:complete} holds. 
By choosing a smaller $r_0>0$, we may assume that $r_0<{\rm dist}(0,\pa U)/4$ (cf. Proposition
\ref{denpro}), that 
\begin{equation} \label{density is Q}
Q+1 \ge \frac{\Ha^n(\Gamma_0 \cap B_{r})}{\omega_n\,r^n} \ge Q \qquad \mbox{with $Q \geq  2$ for all $r \le r_0$}\,,
\end{equation}
and that the growth conditions \eqref{growth}-\eqref{our most affordable g} hold. 
We also fix a small $\zeta>0$ (cf. Definition \ref{chi} and \eqref{cylindrical cut off}) depending only on $Q$ and $n$ such that 
\begin{equation}\label{denlow}
\frac{1}{\omega_n r^n}\int_{\{|T^{\perp}(x)|<\sqrt 2 r\}}\chi_r^2\,d\Ha^n\mres_{\Gamma_0}\geq 
1+\frac{Q-1}{2}
\end{equation}
for all $r\leq r_0$ (by choosing an even smaller $r_0$ if necessary). 

We shall divide the proof into three steps. 
 Throughout the proof we are going to use the following notation. Recall that, given $\lambda > 0$, $\eta_{x,\lambda}$ denotes the function $\eta_{x,\lambda}(y) := \lambda^{-1}\,(y-x)$. 
For the sake of simplicity, we will set $\eta_{\lambda}:=\eta_{0,\lambda}$. 
Furthermore, if $\mathscr{V}=\{V_t\}_{t \ge 0}$ is a family of $n$-varifolds, we will let $\mathscr{W}=(\eta_\lambda)_\sharp \mathscr{V}$ denote the family $\{W_\tau\}_{\tau \ge 0}$ of $n$-varifolds defined by
\begin{equation} \label{brakke flow:bu}
W_\tau := (\eta_{\lambda})_{\sharp} V_{\lambda^2\,\tau}\,,
\end{equation}
where the varifold on the right-hand side is the push-forward of the varifold $V_{\lambda^2\,\tau}$ through the dilation map $\eta_{\lambda}$. It is easy to check by direct calculation that if $\mathscr{V}$ is an $n$-dimensional Brakke flow in $U$ then $(\eta_{\lambda})_\sharp \mathscr{V}$ is an $n$-dimensional Brakke flow in $\eta_{\lambda}(U)=\lambda^{-1}\,U$; see \cite[Section 3.4]{Ton1}.

\subsection{Step one: hole nucleation}
Let $\eps_0$ be given by Lemma \ref{exbrakke}, let $\eps \in \left(0, \eps_0 \right]$, and let $\mathscr{V}^\eps=\{V_t^\eps\}_{t \geq 0}$ be the Brakke flow with fixed boundary $\pa\Gamma_0$ and initial datum $\Gamma_0^\eps$ as in Proposition \ref{prop:brakke_eps}. Correspondingly, consider the Brakke flow 
\[
\hat{\mathscr{V}}^{\eps,1} := (\eta_\eps)_\sharp \mathscr{V}^\eps\,.
\]
By the conclusion in Proposition \ref{prop:brakke_eps}, we have that, denoting $\hat{\mathscr{V}}^{\eps,1}=\{\hat V^{\eps,1}_t\}_{t\ge 0}$,
\begin{equation} \label{rescaled at time zero}
\|\hat V^{\eps,1}_0\| = \Ha^n \mres_{\hat \Gamma_0^{\eps,1}}\,,
\end{equation}
where, as a result of Lemma \ref{exbrakke} and \eqref{growth}-\eqref{our most affordable g}, $\hat \Gamma_0^{\eps,1} := \eps^{-1}\,\Gamma_0^\eps$ satisfies the following properties:
\begin{enumerate}
\item $\Ha^n(\bC(1) \cap U_2 \cap \hat \Gamma_0^{\eps,1}) \leq \omega_n$.
\item $\hat \Gamma_0^{\eps,1} \cap \bC(r_0/\eps)\cap\{|x_{n+1}|<r_0/\eps\} \subset \{(x',x_{n+1})\,\colon\,\abs{x_{n+1}} \leq \frac{\abs{x'}}{\log^\alpha(1/(\eps\abs{x'}))}\}$.
\end{enumerate}
Moreover, \eqref{es-h1} with $R=\eps$ and after rescaling by $\eta_\eps$ gives
\begin{itemize}
\item[(3)] $\bC(\sqrt 2)\cap\{\sqrt 2\leq |x_{n+1}|\leq 2\}\cap {\rm spt}\,\|\hat V_t^{\eps,1}\|=\emptyset$
for $t\in[0,4]$.
\end{itemize}

We apply Lemma \ref{l:exp_holes} to the flow $\{\hat V^{\eps,1}_t\}_{t\geq 0}$ regarded as a 
Brakke flow in $\bC(\sqrt 2)\cap
\{|x_{n+1}|\le 2\}$ with $\hat R_1=\sqrt 2$, $\hat R_2=2$ and 
\begin{eqnarray*}
t_1 = 0 \,, & \qquad & t_2 = 1\,,\\ 
R_1^2 = 1\,, & \qquad & R_2^2=2.
\end{eqnarray*}
We deduce from \eqref{e:almost_mass_monotonicity} as well as (3) that
\begin{equation} \label{s1:estimate}
\begin{split}
2^{-n/2}\, \|\hat V^{\eps,1}_{1}\|(\chi_{\sqrt{2}}^2\mres_{\{|x_{n+1}|\leq 2\}}) & \leq  \|\hat V^{\eps,1}_0\|(\chi_{1}^2\mres_{\{|x_{n+1}|\leq \sqrt 2\}}) + M\,\mu_1^2 \\
& \le \Ha^n (\hat \Gamma_0^{\eps,1} \cap \bC(1)\cap\{|x_{n+1}|\leq \sqrt 2\}) + M\,\mu_1^2\,,
\end{split}
\end{equation}
where in the last inequality we have used \eqref{rescaled at time zero} and the properties of $\chi$, and where
\begin{eqnarray*}
\mu_1^2 = \sup_{t \in \left[0,1\right]} R(t)^{-(n+2)} \,\int_{\bC(R(t))\cap\{|x_{n+1}|\leq \sqrt 2\}} \abs{T^\perp(x)}^2\,d\|\hat V^{\eps,1}_t\|(x)\,, \qquad R(t)^2=1+t\,.
\end{eqnarray*}

Observe that, thanks to property (1), \eqref{s1:estimate} reads
\begin{eqnarray} \label{s1:estimate2}
2^{-n/2} \|\hat V_1^{\eps,1}\|(\chi_{\sqrt{2}}^2\mres_{\{|x_{n+1}|\leq 2\}}) &\leq & \omega_n + M\, \mu_1^2\,.
\end{eqnarray}

Now fix a number $2 \leq L_1 < \infty$ to be chosen later. We can apply Proposition \ref{es-h} (with $R=\eps$ and rescaling by $\eta_\eps$) 
in order to estimate
\begin{eqnarray*}
\mu_1^2 &\leq & \sup_{t\in\left[0,1\right]} \int_{\bC(\sqrt{2})\cap\{|x_{n+1}|\leq \sqrt 2\}} \abs{T^\perp(x)}^2\,d\|\hat V^{\eps,1}_t\|(x)\\
& \leq  & e^{1/4}\,\int_{U_{2L_1}} \abs{T^\perp(x)}^2 \, d\|\hat V^{\eps,1}_0\| (x)
+c(n)\, L_1^{n+2} \, \exp\left( - (L_1-1)^2/8  \right) E_0.
\end{eqnarray*}

We set the following condition: we will choose $L_1$ in such a way that
\begin{equation} \label{assumption_step1}
2\eps L_1 < r_0\,. 
\end{equation}
If \eqref{assumption_step1} holds, then we can use again property (2) above in order to further estimate
\begin{eqnarray*}
\mu_1^2 &\leq &  e^{1/4} \, \frac{(2L_1)^2}{\log^{2\alpha}\left( 1 / (2\eps L_1) \right)}\,\Ha^n(\hat \Gamma_0^{\eps,1} \cap U_{2L_1}) + c(n)\, L_1^{n+2} \, \exp\left( - (L_1-1)^2/8  \right) E_0 \\
& \leq & c(n)\, L_1^{n+2} \left( \frac{1}{\log^{2\alpha}(1 / (2\eps L_1)  )}\, \frac{\Ha^n(\Gamma_0^\eps \cap U_{2L_1\eps})}{(2L_1\eps)^n}  + \exp\left( - (L_1-1)^2/8  \right)E_0\right)\\
& \leq  &   c(n)\, L_1^{n+2} E_0\left( \frac{1}{\log^{2\alpha}(1  /  (2\eps L_1))} + \exp\left( - (L_1-1)^2/8  \right)\right),
\end{eqnarray*}
where we also used \eqref{den-up}. 

Finally, rescaling \eqref{s1:estimate2} back we conclude
\begin{align} \label{s1:estimate rescaled}
& (\sqrt{2}\,\eps)^{-n}\|V^\eps_{\eps^2}\|(\chi^2_{\sqrt{2}\,\eps}\mres_{\{|x_{n+1}|\leq 2\eps\}}) \leq  \omega_n + M\, \mu_1^2\,, \\ \label{error1}
& \mu_1^2  \leq   c(n)\, L_1^{n+2}E_0 \left( \frac{1}{\log^{2\alpha}(1  /  (2\eps L_1))} + \exp\left( - (L_1-1)^2/8  \right)\right)\,.
\end{align}

\subsection{Iteration: hole expansion} Let $h \ge 2$ be an integer, and consider now the Brakke flow
\[
\hat{\mathscr{V}}^{\eps,h} := (\eta_{2^{(h-1)/2}\eps})_\sharp \mathscr{V}^\eps\,.
\]
Again by the conclusions of Proposition \ref{prop:brakke_eps}, and with $\hat{\mathscr{V}}^{\eps,h}=\{\hat V^{\eps,h}_t\}_{t \ge 0}$, we have that
\[
\|\hat V^{\eps,h}_0\| = \Ha^n \mres_{\hat \Gamma^{\eps,h}_0}\,,
\]
where $\hat \Gamma_0^{\eps,h}:= \Gamma_0^\eps / (2^{(h-1)/2}\eps)$ satisfies 
\begin{equation} \label{step_h_flatness}
\begin{split}
\hat \Gamma_0^{\eps,h} \cap \bC\left(\frac{r_0}{2^{(h-1)/2} \eps}\right)& \cap \left\{|x_{n+1}|<\frac{r_0}{2^{(h-1)/2}\eps}\right\} \\ &\subset \left\lbrace (x',x_{n+1}) \, \colon \, \abs{x_{n+1}} \leq \frac{\abs{x'}}{\log^\alpha\left(  1   /   (2^{(h-1)/2}\eps \abs{x'})   \right)} \right\rbrace\,.
\end{split}
\end{equation}
As long as we have
\begin{equation}\label{as-r1}
2^{(h-1)/2}\eps<r_1\,,
\end{equation}
we have by \eqref{es-h1}
\begin{equation}
\bC(\sqrt 2)\cap \{\sqrt 2\leq |x_{n+1}|\leq 2\}\cap{\rm spt}\,\|\hat V_t^{\eps,h}\|=\emptyset\mbox{ for }
t\in[0,4].
\end{equation}
We apply Lemma \ref{l:exp_holes} to the flow $\{\hat V^{\eps,h}_t\}_{t \geq 0}$ with $\hat R_1=\sqrt 2$,
$\hat R_2=2$ and
\begin{eqnarray*}
t_1 = \frac12 \,, & \qquad & t_2 = 1\,,\\ 
R_1^2 = 1\,, & \qquad & R_2^2=2
\end{eqnarray*}
to deduce
\begin{eqnarray} \label{sh:estimate}
2^{-n/2} \|\hat V^{\eps,h}_1\|(\chi_{\sqrt{2}}^2\mres_{\{|x_{n+1}|\leq 2\}}) & \leq & \| \hat V^{\eps,h}_{1/2} \|(\chi_{1}^2\mres_{\{|x_{n+1}|\leq \sqrt 2\}}) + M\, \mu_h^2\,,
\end{eqnarray}
where 
\begin{equation} \label{h:error}
\mu_h^2: = \sup_{t \in \left[ 1/2,1 \right]} R(t)^{-(n+2)} \int_{\bC(R(t))\cap\{|x_{n+1}|\leq \sqrt 2\}} \abs{T^\perp(x)}^2 \, d\|\hat V^{\eps,h}_t\|\,, \,\, R(t)^2=1+2\,\left( t-\frac12\right)\,.
\end{equation}
As long as $L_h$ (to be chosen) satisfies 
\begin{equation}\label{ep-l}
2L_h 2^{(h-1)/2}\eps<r_0,
\end{equation}
by \eqref{es-h6}, 
we have

\begin{eqnarray*}
\mu_h^2 & \leq & \sup_{t \in \left[1/2,1\right]} \int_{\bC(\sqrt{2})\cap\{|x_{n+1}|\leq \sqrt 2\}} \abs{T^\perp(x)}^2 \, d\|\hat V^{\eps,h}_t\| \\
& \leq & e^{1/4}\,\int_{U_{2L_h}} \abs{T^\perp(x)}^2 \, d\|\hat V^{\eps,h}_0\| 
+ c(n)\, L_h^{n+2} \, \exp\left( - (L_h-1)^2/8 \right)E_0.
\end{eqnarray*}

By \eqref{step_h_flatness} and proceeding as in Step one, we further deduce
\begin{eqnarray*}
\mu_h^2 & \leq & c(n) \, L_h^{n+2}E_0 \left( \frac{1}{\log^{2\alpha}\left(    1 / (2^{(h+1)/2}\eps L_h)   \right)}  +  \exp\left( - (L_h-1)^2/8 \right)\right)\,.
\end{eqnarray*}

Now, if we rescale \eqref{sh:estimate} back, we have 
\begin{align} \label{sh:estimate rescaled}
& (2^{h/2}\,\eps)^{-n} \| V^\eps_{2^{h-1}\eps^2}\|(\chi_{2^{h/2}\eps}^2\mres_{\{|x_{n+1}|\leq 2^{(h+1)/2}\eps\}} ) \\ \nonumber
& \quad\leq   (2^{(h-1)/2} \eps)^{-n}    \|V^\eps_{2^{h-2}\eps^2} \|(\chi_{2^{(h-1)/2}\eps}^2\mres_{\{|x_{n+1}|
\leq 2^{h/2}\eps\}})     + M \, \mu_h^2\,,\\ \label{errorh}
& \mu_h^2  \leq  c(n) \, L_h^{n+2} E_0\left( \frac{1}{\log^{2\alpha}\left(    1 / (2^{(h+1)/2}\eps L_h)   \right)}  +  \exp\left( - (L_h-1)^2/8 \right)\right)\,.
\end{align}

\subsection{Conclusion}

Let $j \geq 1$. If we chain the inequalities \eqref{sh:estimate rescaled} as $h$ varies in $\{2,\ldots,j\}$ together with \eqref{s1:estimate rescaled} we conclude that

\begin{equation} \label{after iteration}
(2^{j/2} \, \eps)^{-n} \| V^\eps_{2^{j-1}\eps^2} \|(\chi^2_{2^{j/2}\eps}\mres_{\{|x_{n+1}|\leq 2^{(j+1)/2}\eps\}}) \leq \omega_n + M\, \sum_{h=1}^j \mu_{h}^2\,,
\end{equation}
where, thanks to \eqref{error1} and \eqref{errorh},
\begin{equation} \label{accumulating error}
\mu_h^2 \leq c(n)L_h^{n+2}E_0  \, \left(   \frac{1}{\log^{2\alpha}\left( 1 / (2^{(h+1)/2}\eps L_h)\right)}  +  \exp\left(  - (L_h-1)^2/8  \right)  \right) \,, \qquad h \ge 1\,,
\end{equation}
as long as \eqref{assumption_step1}, \eqref{as-r1} and \eqref{ep-l} are satisfied. In order to guarantee this, we will have to carefully choose $\eps$, $j$, and $L_h$. We proceed as follows.\\

Let $K$ be a large integer, fixed but to be chosen at the end, and for any $J \ge K+1$ apply \eqref{after iteration} with $j=J-K$ and $\eps=\eps_J := 2^{-J/2}$. With these choices, \eqref{after iteration} reads
\begin{equation} \label{after iteration2}
(2^{-K/2})^{-n} \, \|  V^{\eps_J}_{2^{-(K+1)}}  \| (\chi^2_{2^{-K/2}}\mres_{\{|x_{n+1}|\leq 2^{-(K-1)/2}\}} ) \leq \omega_n + M\,\sum_{h=1}^{J-K} \mu_h^2\,.
\end{equation}
We can now choose the constants $L_h$ by setting
\begin{equation} \label{Lh choice}
L_h := \log(J-h)\,, \qquad \mbox{for $1 \leq h \leq J-K$}\,,
\end{equation}
so that \eqref{accumulating error} becomes
\begin{equation} \label{accumulating error 2}
\mu_h^2 \leq c(n)E_0\log^{n+2}(J-h)\, \left(   \frac{1}{\log^{2\alpha}\left(  \frac{2^{(J-h-1)/2}}{\log(J-h)} \right)} + \exp\left( -\frac{(\log(J-h)-1)^2}{8} \right)  \right)\,,
\end{equation}
which is valid assuming that 
\begin{align}
 2^{(h-J+1)/2}\log(J-h)<r_0\,,\label{ass:mdr} \\
 2^{(h-J-1)/2}<r_1 \label{ass:mdr2}
\end{align}
for all $h \in \{ 1,\ldots,J-K\}$, corresponding to \eqref{assumption_step1}, \eqref{as-r1}, and \eqref{ep-l}.  

In order to simplify the notation, it is useful to change variable in the sum from $h$ to $q := J-h$, so that \eqref{after iteration2} becomes
\begin{equation} \label{after iteration final}
(2^{-K/2})^{-n} \, \|  V^{\eps_J}_{2^{-(K+1)}}  \| (\chi^2_{2^{-K/2}}\mres_{\{|x_{n+1}|\leq 2^{-(K-1)/2}\}}) \leq \omega_n + c(n)E_0 M \,\sum_{q=K}^{J-1} a_q^2\,,
\end{equation}
with
\begin{equation} \label{cov}
a_q^2:= \frac{\log^{n+2}(q)}{\log^{2\alpha}\left(  \frac{2^{(q-1)/2}}{\log(q)} \right)} + \log^{n+2}(q)\,\exp\left( -\frac{(\log(q)-1)^2}{8} \right) \,,
\end{equation}
and the conditions \eqref{ass:mdr} and \eqref{ass:mdr2} read
\begin{align}
2^{(-q+1)/2}\log(q)<r_0, \label{am1} \\
2^{(-q-1)/2}<r_1\,\label{am2}
\end{align}
for $q\in\{K,\ldots,J-1\}$. To check the validity of \eqref{am1}, we notice that, for $q$ large, the function $q \mapsto 2^{-(q+1)/2}\log(q)$ is decreasing towards $0$. In particular, \eqref{am1} is satisfied if we choose $K$ large enough depending only on $r_0$. The condition \eqref{am2} is also satisfied as soon as $K$ is large enough depending on $r_1=r_1(n,\alpha,r_0)$.

We have then validated the estimate \eqref{after iteration final} with $a_q^2$ defined by \eqref{cov}. Notice that the estimate remains valid independently of the choice of $J \ge K+1$. Hence, we can now let $J \to \infty$, so that, for a (not relabeled) subsequence of $\{\eps_J\}$ satisfying the conclusion of Proposition \ref{prop:convergence}, and with $\{V_t\}_{t\ge 0}$ the corresponding limit Brakke flow, we have
\begin{equation} \label{limit estimate}
(2^{-K/2})^{-n} \, \|  V_{2^{-(K+1)}}  \| (\chi^2_{2^{-K/2}}\mres_{\{|x_{n+1}|<2^{-(K-1)/2}\}}) \leq \omega_n + c(n)E_0 M\,\sum_{q=K}^{\infty} a_q^2\,.
\end{equation}

Observe that Proposition \ref{prop:convergence} guarantees that $\{V_t\}_{t \ge 0}$ has fixed boundary $\pa \Gamma_0$ and that $\lim_{t \to 0^+} \|V_t\|=\|V_0\|=\Ha^n \mres_{\Gamma_0}$. This shows that $\{V_t\}_{t \ge 0}$ satisfies the conclusion (i) of Theorem \ref{t:main}. Hence, we are only left with proving that $t \mapsto V_t$ is not identically equal to $V_0$. To this end, notice that if $\alpha > \frac12$ then there exists $\gamma > 1$ such that $\lim_{q \to \infty} a_q^2\,q^\gamma=0$, which implies that $\sum_{q=K}^\infty a_q^2$ is a convergent series: therefore, we may choose $K$ so large (depending on $c(n)E_0 M$) that 
\begin{equation} \label{limit estimate final}
 (2^{-K/2})^{-n} \, \|  V_{2^{-(K+1)}}  \| (\chi^2_{2^{-K/2}}\mres_{\{|x_{n+1}|<2^{-(K-1)/2}\}}) \leq \omega_n \,\left( 1 + \frac{Q-1}{4}  \right)\,.
\end{equation}
Due to \eqref{denlow} with $r=2^{-K/2}$ and \eqref{limit estimate final}, 
we see that 
\begin{equation}\label{lef2}
\|V_{r^2/2}\|(\chi^2_{r}\mres_{\{|x_{n+1}|<\sqrt 2 r\}})<\|V_0\|(\chi^2_{r}\mres_{\{|x_{n+1}|<\sqrt 2 r\}}),
\end{equation}
which shows 
$V_{2^{-(K+1)}}\neq V_0$. We may similarly argue that \eqref{lef2} holds for $r=2^{-j/2}$ with any $j>K$. 
Finally, we prove $\|V_t\|(U)<\|V_0\|(U)$ for all $t > 0$. First, note that $\|V_t\|(U)\leq \|V_0\|(U)$ for all $t>0$ by \eqref{brakineq2}. Assume for a contradiction that 
there exists $t_0>0$ with $\|V_{t_0}\|(U)=\|V_0\|(U)$. By \eqref{brakineq2}, for a.e.~$t\in[0,t_0]$,
we have $h(\cdot,V_t)=0$. Choose $j\geq K$ such that $2^{-j}<t_0$. By the above argument, we may choose a 
smooth function $\phi\in C_c^\infty(U)$ with $0\leq \phi\leq 1$ such that $\|V_{2^{-j}}\|(\phi)<\|V_0\|(\phi)$. 
Then, by \eqref{brakineq} and $h(\cdot,V_t)=0$ for a.e.~$t\in[0,t_0]$, we also have $\|V_{t_0}\|(\phi)
<\|V_0\|(\phi)$. Since $\|V_{t_0}\|(U)=\|V_0\|(U)$, we should have $\|V_{t_0}\|(1-\phi)>\|V_0\|(1-\phi)$.
Then by approximation, we have a non-negative function $\hat\phi\in C_c^\infty(U)$
such that $\|V_{t_0}\|(\hat \phi)>\|V_0\|(\hat\phi)$. Since $h(\cdot,V_t)=0$, 
\eqref{brakineq} shows $\|V_{t_0}\|(\hat\phi)\leq \|V_0\|(\hat\phi)$, which is a contradiction.
This shows $\|V_t\|(U)<\|V_0\|(U)$ for all $t>0$ and completes the proof of the theorem. \qed

\bibliographystyle{plain}
\bibliography{MCF_Plateau_biblio}
\end{document}